\documentclass[11pt]{amsart}

\usepackage{multirow}

\usepackage{amsmath,amssymb,amsfonts,amsthm,amscd,amstext,amsxtra,amsopn,array,url,verbatim,mathrsfs,enumerate,anysize,enumitem,xfrac}
\usepackage{graphicx}
\usepackage{amsmath,amssymb,amsfonts,amsthm,amssymb,amscd,url,amstext,amsxtra,amsopn
,txfonts}
\usepackage{verbatim}
\marginsize{2.25cm}{2.25cm}{2.25cm}{2.25cm}
\usepackage[dvipsnames,usenames,table,xcdraw]{xcolor}
\usepackage[colorlinks=true,urlcolor=Black,citecolor=Black,linkcolor=Black]{hyperref}
\usepackage{times}
\usepackage{mathabx}

\usepackage[tableposition=top]{caption}
\usepackage{verbatim}
\usepackage{enumerate}
\usepackage{float}    \floatstyle{plaintop} \restylefloat{table} 
\usepackage{graphicx}
\usepackage{epstopdf}
\usepackage{ upgreek }
\usepackage{esdiff}
\usepackage{mathrsfs}
\usepackage[section]{placeins}

\usepackage[acronym]{glossaries}

\usepackage{nomencl}
\makenomenclature

\usepackage{mathtools}

\usepackage{dsfont}
\usepackage{latexsym}
\usepackage{listings}
\usepackage{multirow}
\usepackage{soul}
\usepackage{subfigure}
\usepackage{url}
\usepackage{xspace}
\usepackage{tikz}

\newcommand{\bigO}{\mathcal{O}}

\renewcommand{\bar}[1]{\overline{#1}}    

\renewcommand{\phi}{\varphi}

\makeatletter
\def\imod#1{\allowbreak\mkern5mu({\operator@font mod}\,\,#1)}
\makeatother

\makeatletter
\def\starmod#1{\allowbreak\mkern5mu({\operator@font mod^{*}}\,\,#1)}
\makeatother

\definecolor{Gray}{gray}{0.9}

\makeatletter
\@namedef{subjclassname@2010}{%
  \textup{2010} Mathematics Subject Classification}
\makeatother

\newtheorem*{rep@theorem}{\rep@title}
\newcommand{\newreptheorem}[2]{%
\newenvironment{rep#1}[1]{%
 \def\rep@title{#2 \ref{##1}}%
 \begin{rep@theorem}}%
 {\end{rep@theorem}}}
\makeatother

\newtheorem{thm}{Theorem}[section]
\newtheorem{theorem}[thm]{Theorem}
\newtheorem{proposition}[thm]{Proposition}

\newtheorem{lemma}[thm]{Lemma}
\newtheorem{corollary}[thm]{Corollary}

\newtheorem{question}[thm]{Question}

\newtheorem{hypothesis}[thm]{Hypothesis}
\newtheorem{criterion}[thm]{Criterion}

\newtheorem*{theorem*}{Theorem}
\newtheorem*{problem*}{Problem}
\newtheorem*{corollary*}{Corollary}

\theoremstyle{definition}

\newtheorem{remarks}[thm]{Remarks}
\newtheorem{notation}[thm]{Notation}

\newcommand{\normalarray}{\renewcommand{\arraystretch}{1.1}}

\normalarray

\newcommand{\C}{\mathbb{C}} 
\newcommand{\N}{\mathbb{N}} 
\newcommand{\Q}{\mathbb{Q}} 
\newcommand{\R}{\mathbb{R}} 
\newcommand{\Z}{\mathbb{Z}} 

 {\begin{list}{}%
    {\setlength{\leftmargin}{#1}}%
  \item[]%
 }
{\end{list}}

\usepackage[citation-order]{amsrefs}
\newenvironment{rezabib}
  {\bibdiv\biblist\setupbib}
  {\endbiblist\endbibdiv}

\def\setupbib{\catcode`@=\active}
\begingroup\lccode`~=`@
  \lowercase{\endgroup\def~}#1#{\gatherkey{#1}}
\def\gatherkey#1#2{\gatherkeyaux{#1}#2\gatherkeyaux}
\def\gatherkeyaux#1#2,#3\gatherkeyaux{\bib{#2}{#1}{#3}}

\begin{document}

\title{Euler's Function on Products of Primes in Progressions}

\thanks{Research of both authors is partially supported by NSERC}

\date{\today}



\author{Amir Akbary}
\author{Forrest J. Francis}

\address{Department of Mathematics and Computer Science \\
        University of Lethbridge \\
        Lethbridge, AB T1K 3M4 \\
        Canada}
        \email{amir.akbary@uleth.ca}
 \address{Department of Mathematics and Computer Science \\
        University of Lethbridge \\
        Lethbridge, AB T1K 3M4 \\
        Canada}
\email{fj.francis@uleth.ca}

\begin{abstract}
We study generalizations of some results of Jean-Louis Nicolas regarding the relation between small values of Euler's function $\varphi(n)$ and the Riemann Hypothesis. Among other things, we prove that for $1\leq q\leq 10$ and for $q=12, 14$, the generalized Riemann Hypothesis for the Dedekind zeta function of the cyclotomic field $\mathbb{Q}(e^{2\pi i/q})$ is true if and only if for all integers $k\geq 1$ we have
\[\frac{\bar{N}_k}{\varphi(\bar{N}_k)(\log(\varphi(q)\log{\bar{N}_k}))^{\frac{1}{\varphi(q)}}} > \frac{1}{C(q,1)}.\]
Here $\bar{N}_k$ is the product of the first $k$ primes in the arithmetic progression $p\equiv 1 \imod{q}$  and 

$C(q, 1)$ is the constant appearing in the asymptotic formula
 \[\prod_{\substack{p \leq x \\ p \equiv 1\imod{q}}} \left(1 - \frac{1}{p}\right) \sim \frac{C(q, 1)}{(\log{x})^\frac{1}{\varphi(q)}},\]
 as $x\rightarrow\infty$. We also prove that, for $q\leq 400,000$ and integers $a$ coprime to $q$, the analogous inequality 
 \[\frac{\bar{N}_k}{\varphi(\bar{N}_k)(\log(\varphi(q)\log{\bar{N}_k}))^{\frac{1}{\varphi(q)}}} > \frac{1}{C(q,a)}\]
 holds for infinitely many values of $k$. If in addition $a$ is a not a square modulo $q$, then there are infinitely many $k$ for which this inequality holds and also infinitely many $k$ for which this inequality fails.
\end{abstract}

\maketitle

\section{Introduction}

Let $\varphi(n)$ be Euler's totient function. A result of Landau from 1909 captures the minimal behavior of $\frac{\varphi(n)}{n}$. 

\begin{theorem}[{\cite{hardy2008}*{Theorem 328}}]\label{landauthm} Let $\gamma$ be the Euler-Mascheroni constant. Then
\[\limsup_{n \rightarrow \infty} \frac{n}{\varphi(n) \log{\log{n}}} = e^\gamma.\]
\end{theorem}
A proof of Theorem \ref{landauthm} follows by considering the sequence of \emph{primorials}, $N_k := \prod_{i=1}^k p_i$, alongside Mertens' theorem \cite[Theorem 429]{hardy2008} and the Prime Number Theorem \cite[Theorem 6]{hardy2008}. Theorem \ref{landauthm} is visually expressed in Figure \ref{figure:landauthm}. (Colored plots throughout this paper have been generated using Maple${}^\mathrm{TM}$ \footnote{Maple is a trademark of Waterloo Maple, inc.} \cite{maple}.)

\begin{figure}[h]
\centering
\includegraphics[width=0.4\textwidth,height=0.3\textheight,keepaspectratio]{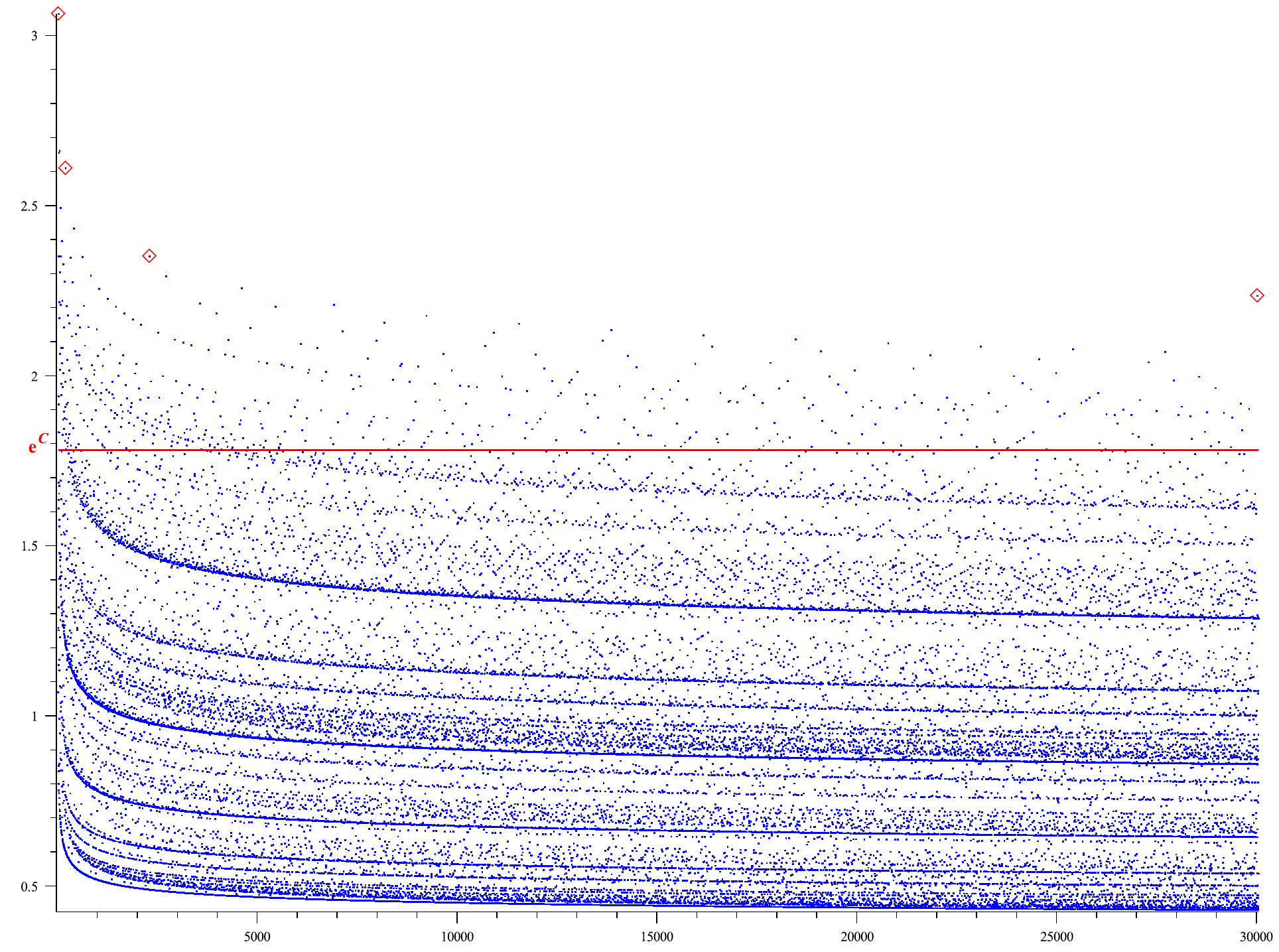}
\caption[Visualization of Theorem \ref{landauthm}]{Visualization of Theorem \ref{landauthm}. Points represent $(n,\tfrac{n}{\varphi(n) \log{\log{n}}})$ for $n$ between 29 and 30055. Points marked with red diamonds correspond to the primorials 30, 210, 2310, and 30030.}
\label{figure:landauthm}
\end{figure}

In  \cite{rosser1962},  Rosser and Schoenfeld studied the behavior of the expresssion $\tfrac{n}{\varphi(n)\log\log{n}}$ in a more explicit manner.
\begin{theorem}[{\cite[Theorem 15]{rosser1962}}]
 For $n > 2$, 
\[\frac{n}{\varphi(n) \log{\log{n}}} \leq e^\gamma + \frac{2.50637}{(\log\log{n})^2}.\]
\end{theorem}
Rosser and Schoenfeld also remarked that they do not know whether there are infinitely many natural numbers $n$ satisfying
\begin{equation}\label{NicolasIneq}
 \frac{n}{\varphi(n) \log{\log{n}}} > e^\gamma.
\end{equation}

In \cite{nicolas1983}, Nicolas expressed the preceding remark by asking the following question. 

\begin{question}[{\cite[p. 375]{nicolas1983}}]\label{RSQ}
 Do there exist infinitely many $n \in \N$ for which 
$\tfrac{n}{\varphi(n)\log\log{n}} > e^\gamma?$
\end{question}

In the same paper, he resolves this question in the following theorem. 

\begin{theorem}[{\cite[Theorem 1]{nicolas1983}}]\label{nicolas1}
 There exist infinitely many $n \in \N$ for which $\frac{n}{\varphi(n) \log\log{n}} > e^\gamma$.
\end{theorem}

Nicolas' proof leverages properties of the Riemann zeta function, $\zeta(s)$, against the behavior of $\varphi(n)$ at primorials. Recall that the Riemann zeta function has a pole at $s = 1$ and trivial zeroes at $s = {-2}, {-4}, {-6}\ldots$. Its \emph{nontrivial} zeroes are those found in the \emph{critical strip} $0<\Re(s)<1$.  The Riemann Hypothesis (RH) predicts the location of these nontrivial zeroes.
\begin{hypothesis}[{RH}]
The nontrivial zeroes of $\zeta(s)$ have real part $1/2$.
\end{hypothesis}

In relation to Question \ref{RSQ}, Nicolas considered the behavior of $\varphi(N_k)$ under two possible resolutions to the Riemann Hypothesis. 
\begin{theorem}[{\cite[Theorem 2]{nicolas1983}}] \label{Nicolas83} If the Riemann Hypothesis is true, then for all \emph{primorials} $N_k = \prod_{i=1}^k p_i$, where $p_i$ is the $i$-th prime, we have
\[ \frac{N_k}{\varphi(N_k) \log{\log{N_k}}} > e^\gamma.\]
If the Riemann Hypothesis is false, then there are infinitely many primorials for which the above inequality holds and also infinitely many primorials for which the above inequality does not hold. 
\end{theorem}
As a direct corollary of the above theorem, we have the following criterion for the Riemann Hypothesis.

\begin{criterion}[Nicolas' Criterion for the RH]\label{NicolasRH}
 The Riemann Hypothesis is true if and only if  there exists $k_0 >0$ such that for all $k\geq k_0$,
\[ \frac{N_k}{\varphi(N_k) \log{\log{N_k}}} > e^\gamma.\]
\end{criterion}

In this article, we are motivated to generalize the above criterion in the context of Dirichlet $L$-functions. Towards this aim we seek  a generaliziation of Theorem \ref{Nicolas83} to a setting involving primes in arithmetic progressions. We will restrict our attention to the behavior of $\varphi(n)$ at elements of the set 
 \[S_{q, a} := \left\{n \in \N \mbox{ ; } p \mid n \implies p \equiv a \imod{q}\right\},\]
 where $q, a$ are coprime natural numbers. This set contains an analogue of the primorials, the $k$-th \emph{primorial in} $S_{q,a}$ given by
 \[\bar{N}_k =N_{q,a}(k) := \prod_{i=1}^k \bar{p}_i,\]
where $\bar{p}_i$ is the $i$-th prime in the arithmetic progression $a \imod{q}$. Throughout this article, $q$ and $a$ will be fixed and coprime. For notational convenience, we often suppress reference to $q$ and $a$ and use $\bar{N}_k$ to denote the $k$-th primorial in $S_{q,a}$. (Note that for $q=a=1$, we have $\bar{N}_k = N_k$.) In this context, we have analogues for both Mertens' theorem and the Prime Number Theorem. The analogue of Mertens' theorem was originally established by Williams \cite{williams1974}. Here, however, we refer to the work of Languasco and Zaccagnini (\cite{languasco2007}, \cite{languasco2009}, \cite{languasco2010}) where they have provided an explicit form for the constant appearing in the generalized Mertens' theorem.

\begin{theorem}[{\cite[p. 46]{languasco2007}}]\label{mertensap}
Let $x \geq 2$ and $q, a \in \N$ be coprime. Then, 
\[\prod_{\substack{p \leq x \\ p \equiv a\imod{q}}} \left(1 - \frac{1}{p}\right) \sim \frac{C(q, a)}{(\log{x})^\frac{1}{\varphi(q)}},\]
as $x \rightarrow \infty$, where 
\[C(q, a)^{\varphi(q)} = e^{-\gamma} \prod_{p} \left(1 - \frac{1}{p}\right)^{\alpha(p; q, a)} \]
and 
\[ \alpha(p; q, a) = 
\begin{cases}
\varphi(q) - 1 & \textnormal{if } p \equiv a \imod{q}, 
\\
-1 &\textnormal{otherwise. }
\end{cases}
\]
\end{theorem}
We note that, in agreement with the classical Mertens' theorem, $C(1,1)$ is $e^{-\gamma}$ since $\alpha(p;1,1) = 0$ for all primes $p$. 

For an analogue of the prime number theorem, we have
\begin{equation}\label{pntap}
\theta(x; q, a) \sim \frac{x}{\varphi(q)},
\end{equation}
as $x \rightarrow \infty$ (see {\cite[Theorem 6.8]{narkiewicz2000}}), where
\[\theta(x; q, a) := \sum_{\substack{p \leq x \\ p \equiv a\imod{q}}} \log{p}.\]

Hence, we have all of the tools required to establish a generalization of Theorem \ref{landauthm} for primes in arithmetic progressions. 

\begin{theorem}\label{app}
Let $q, a \in \N$ be coprime. Then
\[\limsup_{n \in S_{q, a}}  \frac{n}{\varphi(n) (\log(\varphi(q)\log{n}))^{1/\varphi(q)}} = \frac{1}{C(q,a)}, \]
where $C(q, a)$ is defined in Theorem \ref{mertensap}. 
\end{theorem}

\begin{figure}[h]
  \centering
  \begin{minipage}[t]{0.49\textwidth}
    \includegraphics[width=\textwidth]{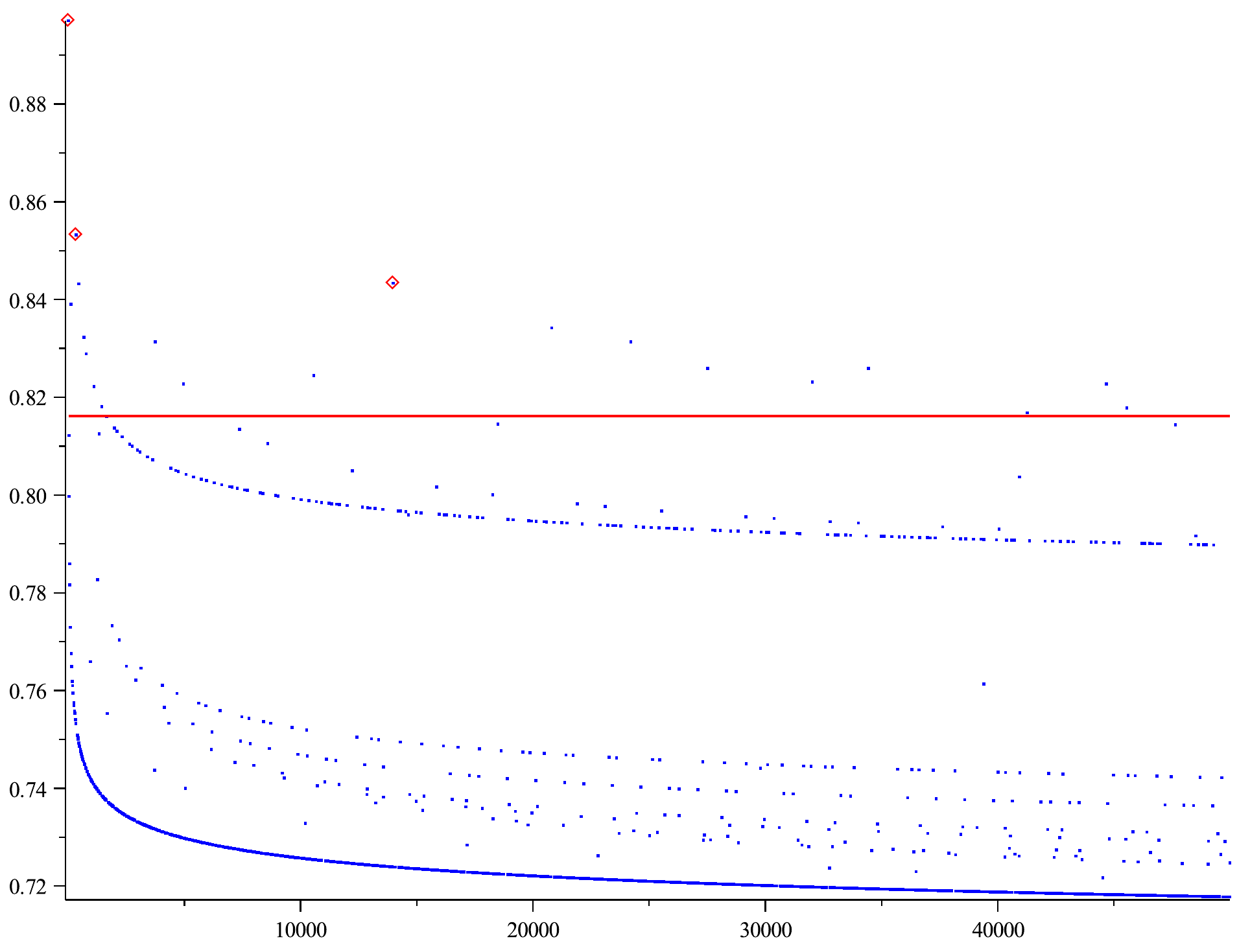}
\caption[Visualization of Theorem \ref{app}]{Visualization of Theorem \ref{app} for $q = 5$, $a = 1$. Points represent $(n,\tfrac{n}{(\varphi(n) \log{\log{n}})^\frac{1}{\phi(q)}})$ for $n \in S_{5,1}$ between 11 and 49991. Points marked with red diamonds correspond to $11$, $341$, and $13981$, the first three primorials in $S_{5,1}$. The red line is $C(5,1)^{-1} \approx 1.2252$.}\label{figure:filter51}
  \end{minipage}
  \hfill
  \begin{minipage}[t]{0.49\textwidth}
    \includegraphics[width=\textwidth]{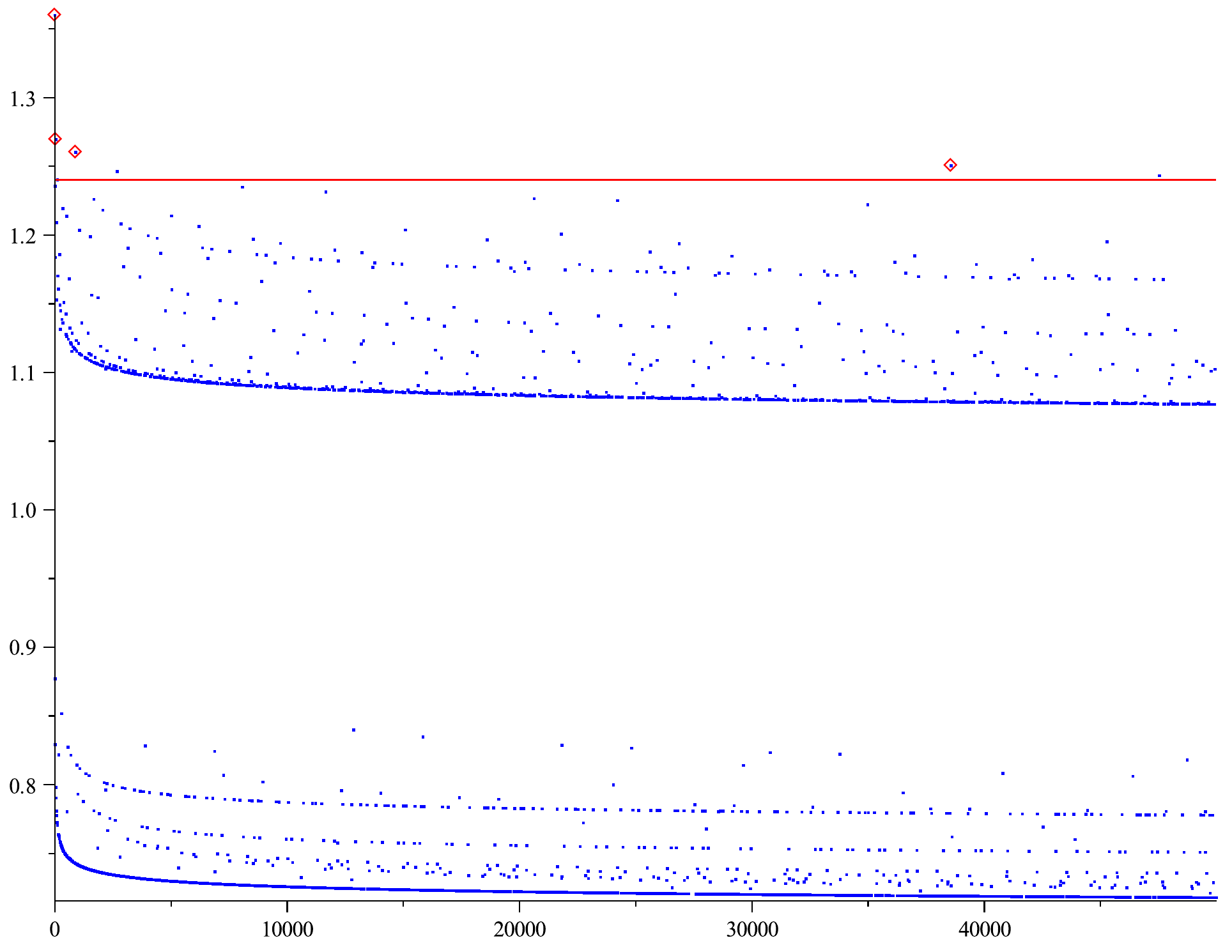}
\caption[Visualization of Theorem \ref{app}]{Visualization of Theorem \ref{app} for $q = 5$, $a = 3$. Points represent $(n,\tfrac{n}{(\varphi(n) \log{\log{n}})^\frac{1}{\phi(q)}})$ for $n \in S_{5,3}$ between 3 and 49993. Points marked with red diamonds correspond to $3, 39, 897$, and $38571$, the first four primorials in $S_{5,3}$. The red line is $C(5,3)^{-1} \approx 0.8060$.}\label{figure:filter53}
  \end{minipage}
\end{figure}

In each fixed set $S_{q,a}$, one can observe behavior congruous with the behavior in Figure \ref{figure:landauthm}. Figures \ref{figure:filter51} and \ref{figure:filter53} visualize Theorem \ref{app} for $q=5$ and $a = 1, 3$. 

At this point, it seems reasonable to extend this generalization along the line of study begun by Nicolas. We are interested in the following question.
\begin{question}\label{GQ}
 Let $q, a \in \N$ be coprime and consider the inequality 
 \begin{equation}\label{GeneralIneq}
  \frac{n}{\varphi(n)(\log(\varphi(q)\log{n}))^{\frac{1}{\varphi(q)}}} > \frac{1}{C(q,a)}.
 \end{equation}
Are there infinitely many $n \in S_{q,a}$ for which \eqref{GeneralIneq} is satisfied?
\end{question}

Nicolas \cite[pp. 376-77]{nicolas1983} observed that one can encode information regarding \eqref{GeneralIneq} at primorials in $S_{1,1}$, using a real-valued function. Mimicking his construction, let $\bar{p}$ represent any prime in the progression $a \imod{q}$. Define 
\[f(x;q,a) := \frac{{\left(\log(\varphi(q){\theta(x;q,a)})\right)}^{\frac{1}{\varphi(q)}}}{C(q,a)}\cdot \prod_{\bar{p}\leq x} \left(1-\frac{1}{\bar{p}}\right).\]
Hence, for any $x\in \left[\bar{p}_k, \bar{p}_{k+1}\right)$, 
\[f(x;q,a) = \frac{{\left(\log(\varphi(q)\log{\bar{N}_k})\right)}^{\frac{1}{\varphi(q)}}}{C(q,a)}\cdot\frac{\varphi(\bar{N}_k)}{\bar{N}_k}.\]
It is therefore apparent that \eqref{GeneralIneq} holds for $\bar{N}_k$ if and only if $f(x;q,a) < 1$ for any $x \in \left[\bar{p}_k, \bar{p}_{k+1}\right)$ or, equivalently,
\begin{equation}\label{equivalence}
 \log{f(x;q,a)} = \frac{\log\log(\varphi(q){\theta(x;q,a)})}{\varphi(q)} + \sum_{\bar{p} \leq x} \log\left(1-\frac{1}{\bar{p}}\right) - \log{C(q,a)} < 0,
\end{equation}
for $x \in \left[\bar{p}_k, \bar{p}_{k+1}\right)$.

We supply plots of $\log f(\bar{p}_k;q,a)$ for several values of $q$ and $a$. Since $\log f(x;q,a)$ is fixed between primes in the progression $a \imod{q}$, the horizontal axis in each plot is $k$, the index of $\bar{p}_k$, rather than $x$. Each plot presents data for primes $\bar{p}_k < 50,000$. For example, the first plot indicates that $f(x;1,1) < 1$ for $x \leq 49,999$. In Theorem 3(a) of \cite{nicolas1983}, Nicolas showed (using estimates of Rosser and Schoenfeld \cite{rosser1962}) that $f(x;1,1) < 1$ for $2 \leq x \leq 10^8$ and further that, assuming RH, it will remain negative for all values of $x$. The plots distinguish between three cases of residues modulo $q$. For $a = 1$, $\log{f(p_k;q,a)}$ is black; for other square $a$, $\log{f(p_k;q,a)}$ is red or yellow; for non-square $a$, $\log{f(p_k;q,a)}$ is a cool color. These plots suggest that the behavior of $\log{f(p_k;q,a)}$ around 0 differs depending on whether $a$ is a square or a non-square modulo $q$. Of note is the plot for $q = 7$, where we come across several examples where $\log{f(\bar{p}_k;7,a)} > 0$, all of which occur when $a$ is not square modulo 7.  

\begin{figure}[!bp]
  \centering
  \begin{minipage}[t]{0.49\textwidth}
    \includegraphics[width=\textwidth]{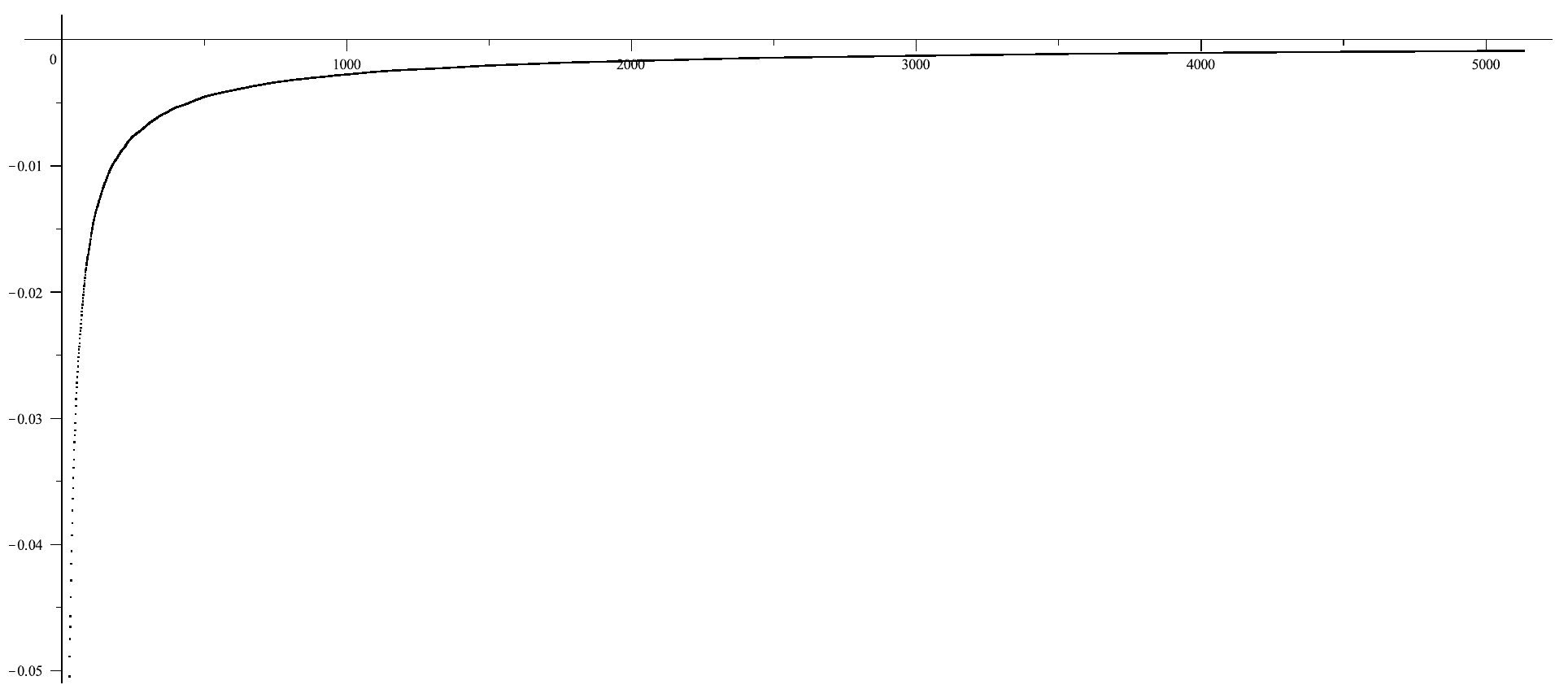}
\caption[Plot of $\log{f(p_k;1,1)}$]{Plot of $\log{f(p_k;1,1)}$.}
  \end{minipage}
  \hfill
  \begin{minipage}[t]{0.49\textwidth}
    \includegraphics[width=\textwidth]{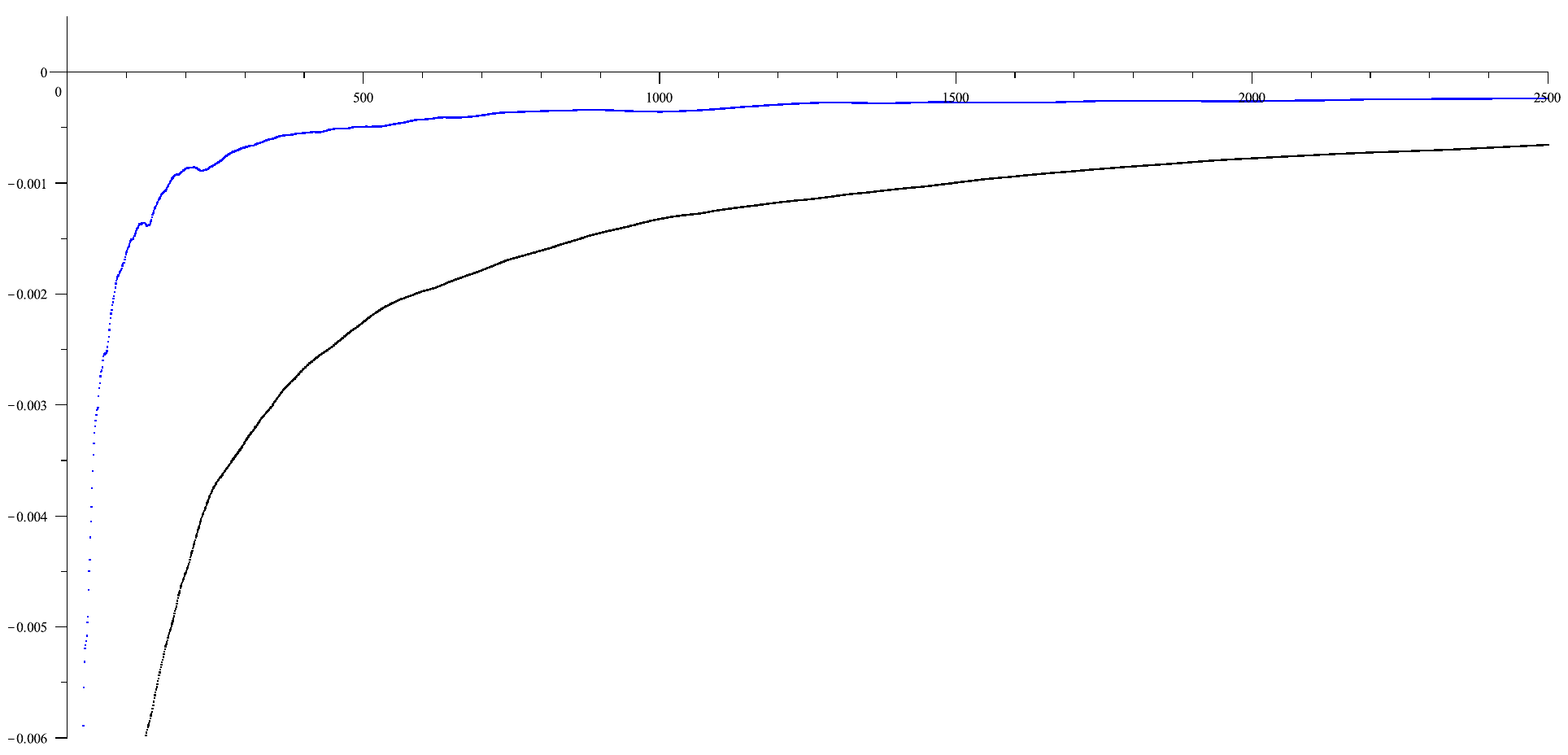}
\caption[Plot of $\log{f(\bar{p}_k;3,a)}$]{Plot of $\log{f(\bar{p}_k;3,a)}$. \newline The black plot corresponds to $a = 1$ and the blue plot to $a = 2$.} 
  \end{minipage}
\end{figure}

\begin{figure}[!tbp]
  \centering
  \begin{minipage}[t]{0.49\textwidth}
    \includegraphics[width=\textwidth]{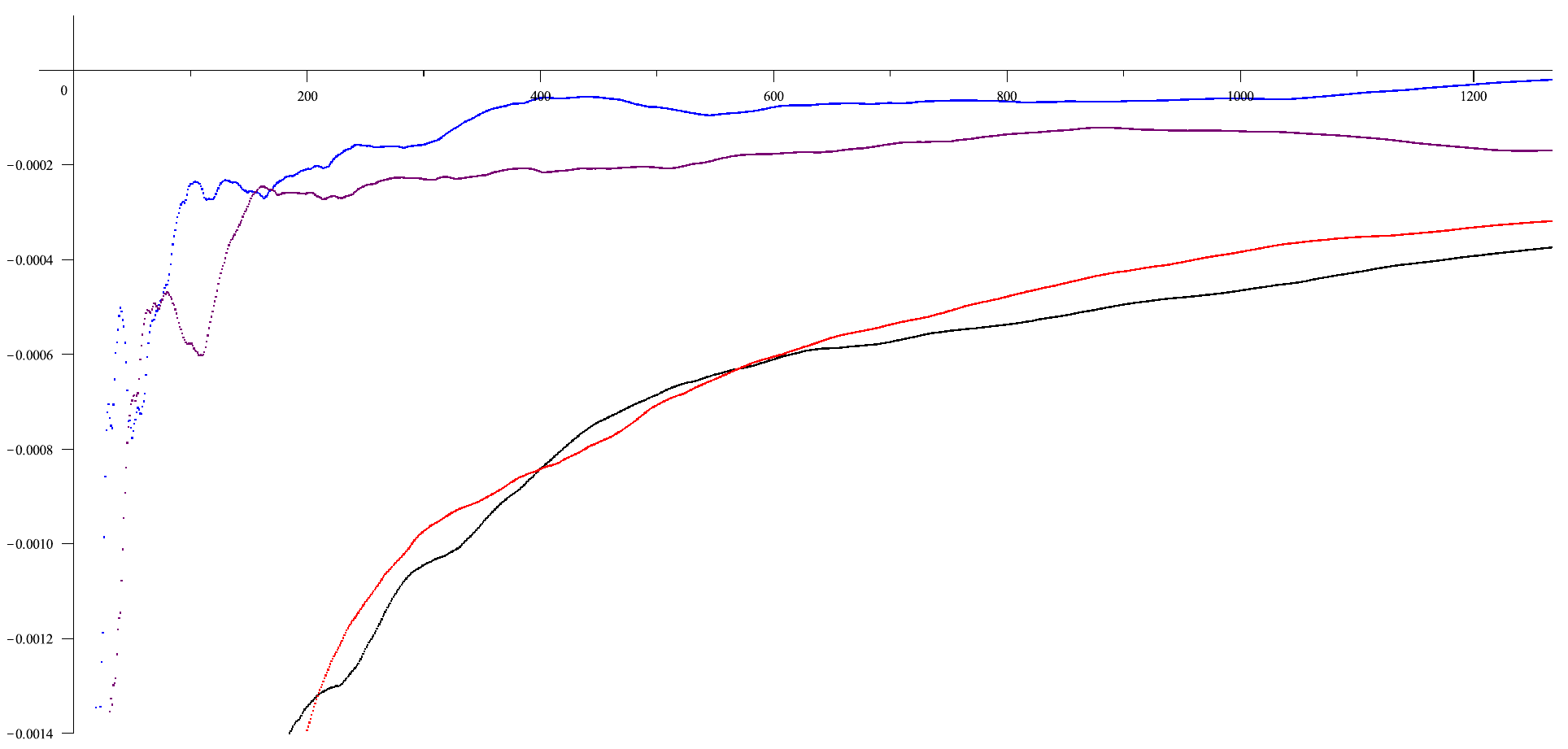}
\caption[Plot of $\log{f(\bar{p}_k;5,a)}$]{Plot of $\log{f(\bar{p}_k;5,a)}$. The black plot corresponds to $a = 1$, the red plot to $a = 4$, the blue plot to $a = 2$, and the purple plot to $a = 3$.} 
  \end{minipage}
  \hfill
  \begin{minipage}[t]{0.49\textwidth}
    \includegraphics[width=\textwidth]{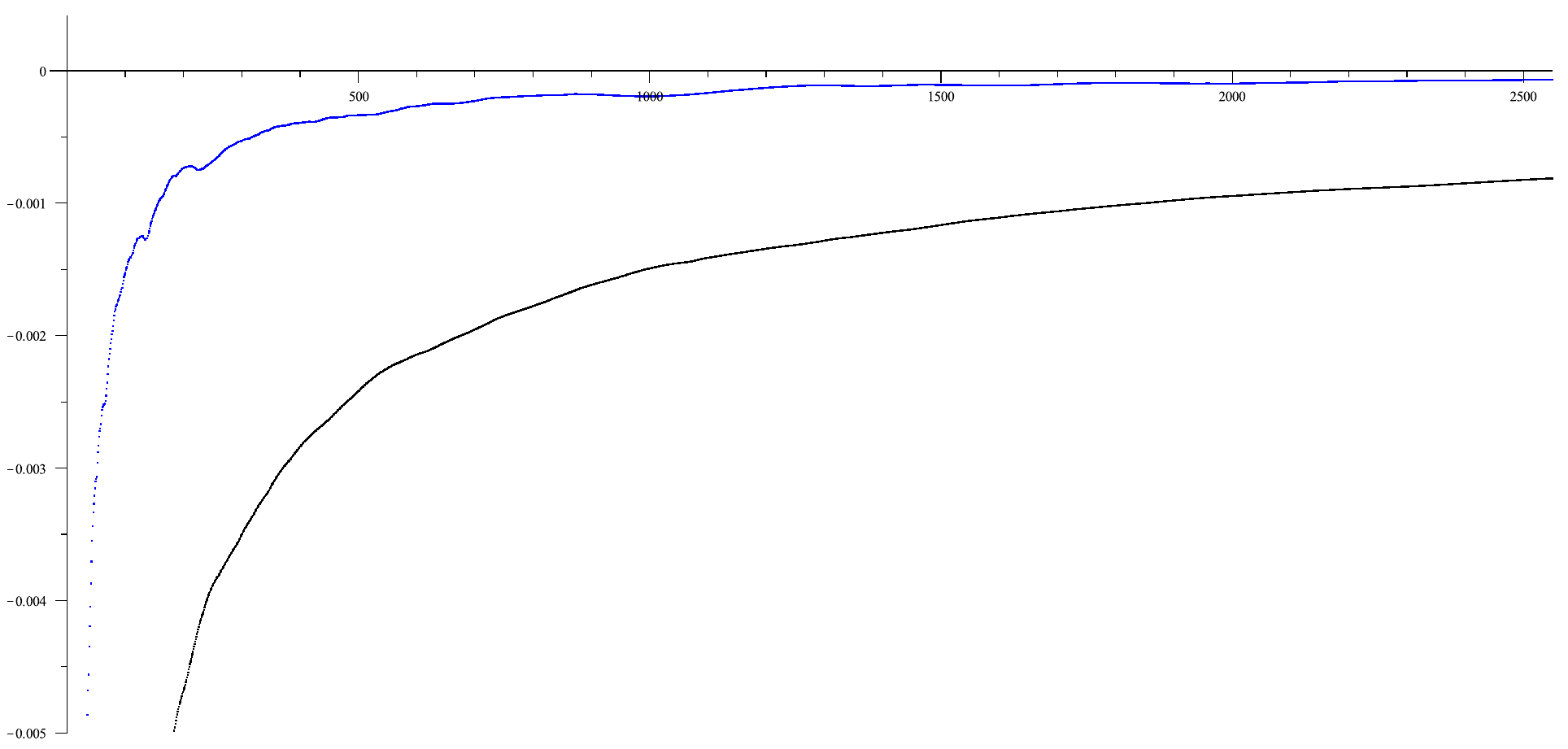}
\caption[Plot of $\log{f(\bar{p}_k;6,a)}$]{Plot of $\log{f(\bar{p}_k;6,a)}$. The black plot corresponds to $a = 1$ and the blue plot to $a = 5$.} 
  \end{minipage}
\end{figure}

\begin{figure}[!tbp]
  \centering
  \begin{minipage}[t]{0.49\textwidth}
    \includegraphics[width=\textwidth]{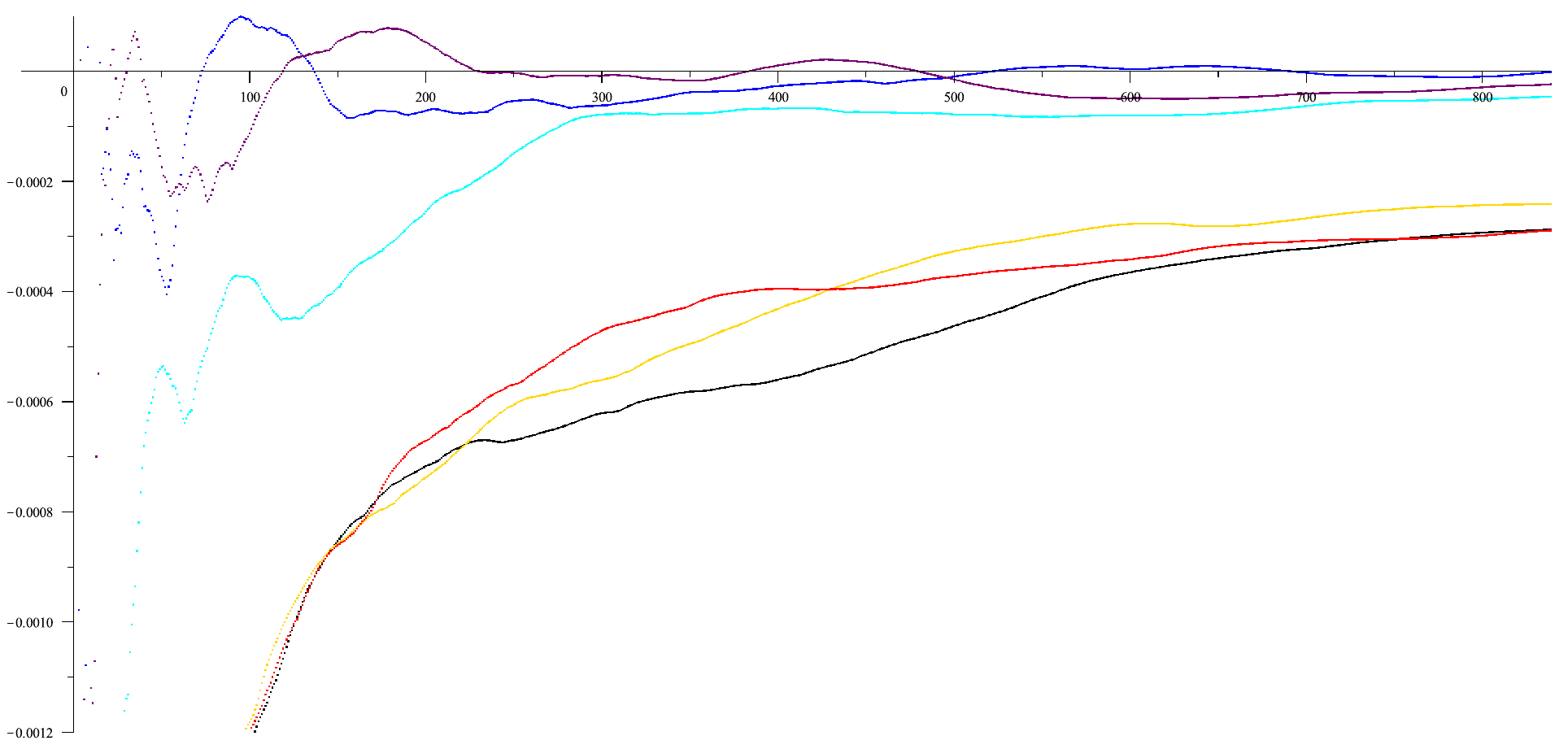}
\caption[Plot of $\log{f(\bar{p}_k;7,a)}$]{Plot of $\log{f(\bar{p}_k;7,a)}$. The black plot corresponds to $a = 1$, the yellow plot to $a = 2$, the red plot to $a = 4$, the blue plot to $a = 3$, the purple plot to $a = 5$, and the cyan plot to $a = 6$.} 
  \end{minipage}
  \hfill
  \begin{minipage}[t]{0.49\textwidth}
    \includegraphics[width=\textwidth]{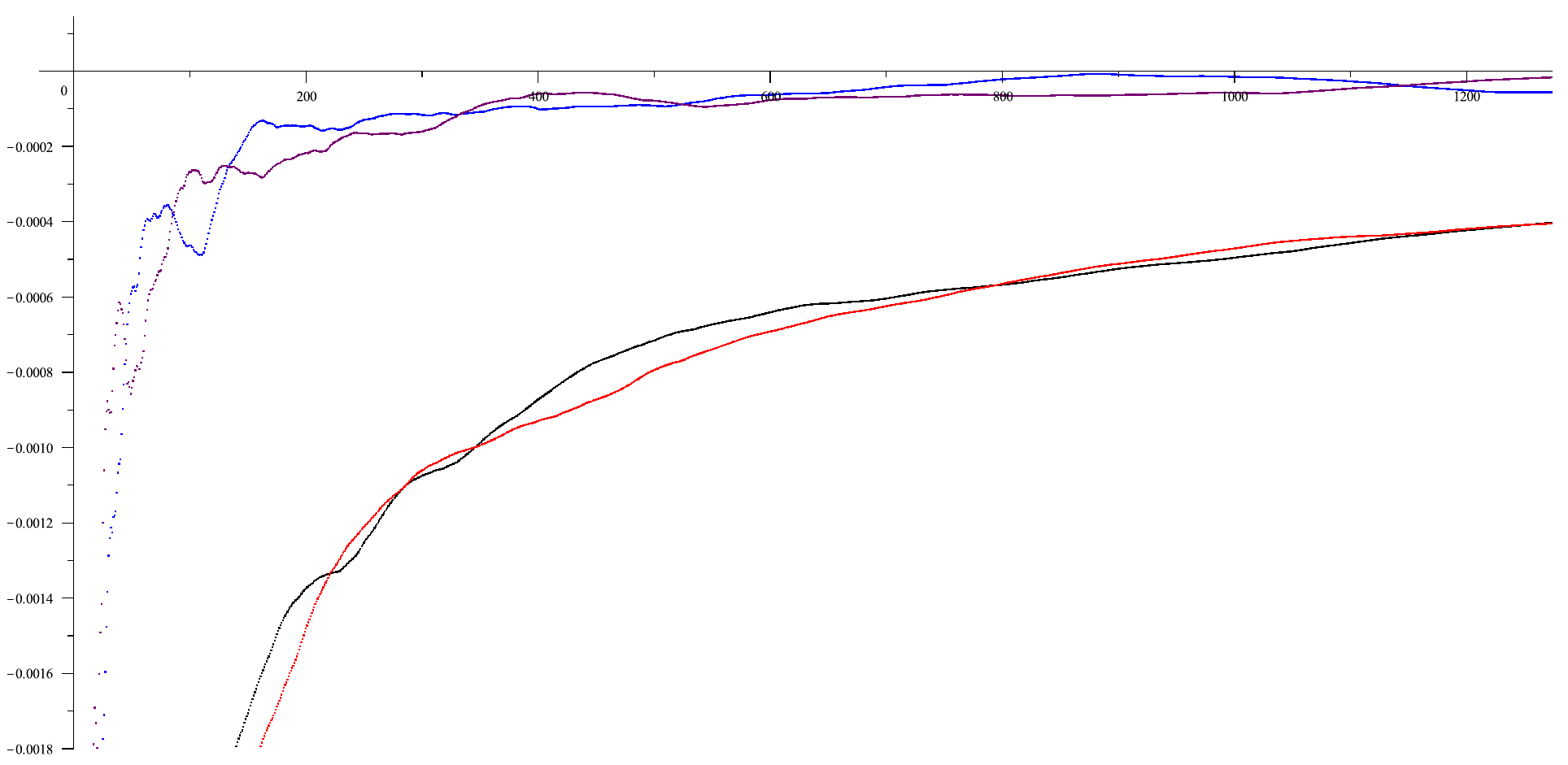}
\caption[Plot of $\log{f(\bar{p}_k;10,a)}$]{Plot of $\log{f(\bar{p}_k;10,a)}$. The black plot corresponds to $a = 1$, the red plot to $a = 9$, the blue plot to $a = 3$, and the purple plot to $a = 7$.} 
  \end{minipage}
\end{figure}

In this paper we study Question \ref{GQ} by examining the function $\log{f(x;q, a)}$ along the line of approach developed by Nicolas for the case $q=1$ in \cite{nicolas1983}. Among other results we propose an answer regarding the different behaviour of $\log{f(x;q, a)}$ when $a$ is not a square modulo $q$. We will study the function $\log f(x;q,a)$ by appealing to the behavior of Dirichlet $L$-functions $L(s,\chi)$ arising from Dirichlet characters $\chi \imod{q}$. As with the Riemann zeta function, we say that the nontrivial zeroes of $L(s, \chi)$ are located on the critical strip $0<\Re(s)<1$. The relevant analogue of the Riemann Hypothesis for Dirichlet $L$-functions mod $q$ is the following.

\begin{hypothesis}[{{GRH}${}_q$}]\label{GRHq}
 For fixed $q$ and all Dirichlet characters $\chi$  modulo $q$, all
 the nontrivial zeroes of $L(s, \chi)$ have real part $1/2$.
\end{hypothesis}
It is known that $$\zeta_{{\mathbb{Q}}(e^{2\pi i/q})}(s)=\prod_{\chi \imod{q}} L(s, \chi^\prime),$$ where $\zeta_{{\mathbb{Q}}(e^{2\pi i/q})}(s)$ is the Dedekind zeta function of the cyclotomic field ${{\mathbb{Q}}(e^{2\pi i/q})}$ and $\chi^\prime$ is the primitive Dirichlet character that induces the Dirichlet character $\chi$ modulo $q$ (see \cite[Theorem 65, p. 296]{FM}). Since the nontrivial zeroes of $\zeta_{{\mathbb{Q}}(e^{2\pi i/q})}(s)$ are located in the critical strip $0<\Re(s)<1$, Hypothesis \ref{GRHq} (GRH${}_q$) is equivalent to  the statement of the Generalized Riemann Hypothesis (GRH) for $\zeta_{{\mathbb{Q}}(e^{2\pi i/q})}(s)$ (i.e., all the nontrivial zeroes of $\zeta_{{\mathbb{Q}}(e^{2\pi i/q})}(s)$ are located on the line $\Re(s)=1/2$). Also note that GRH${}_q$ is equivalent to the statement that all the singularities of the function 
 \[\mathscr{L}(s;q,1) = \sum_{\chi \imod{q}}  \frac{L'}{L}(s,\chi)\]
 on the critical strip are located on the line $\Re(s)=1/2$.

An important function in this article is the following linear combination of logarithmic derivatives of the Dirichlet $L$-functions $L(s, \chi)$. For $q$ and $a$, fixed coprime integers, we set
 \[\mathscr{L}(s;q,a) = \sum_{\chi \imod{q}} \overline{\chi}(a) \frac{L'}{L}(s,\chi).\]

The function $\mathscr{L}(s;q,a)$ has a potential singularity at $s$ when $s$ is a nontrivial zero of a Dirichlet $L$-function corresponding to a character modulo $q$. We formulate the following Singularity Hypothesis (SH${}_{q, a}$) for $\mathscr{L}(s;q,a)$.

\begin{hypothesis}[SH${}_{q, a}$]\label{zeroconj}
 Let $q$ and $a$ be fixed coprime integers. Then the singularities of  $\mathscr{L}(s;q,a)$ on the critical strip $0<\Re(s)<1$ are located on the line $\Re(s)=1/2$.
\end{hypothesis}
Note that GRH${}_q$ implies SH${}_{q, a}$, however SH${}_{q, a}$ for $a\neq 1$  does not eliminate the existence of the nontrivial zeroes $\rho$ of $L(s, \chi)$ with $\Re(\rho)\neq 1/2$.  More precisely, SH${}_{q, a}$ implies that if there exists a nontrivial zero $\rho$ of $L(s,\chi)$ for some $\chi$ modulo $q$ for which $\Re(\rho)$ is not ${1}/{2}$, then
\[
  \sum_{\chi \imod{q}} \overline{\chi}(a) m_\rho(\chi) = 0,
\]
where $m_\rho(\chi)$ is the multiplicity of the zero $\rho$ of $L(s,\chi)$. For $a = 1$, SH${}_{q, 1}$ is equivalent to GRH${}_{q}$.

It comes as a consequence of the work of Platt \cite[Theorem 10.1, Theorem 10.2]{platt2016} that for $L$-functions associated with primitive Dirichlet characters of modulus $q$ less than 400,000, there are no zeroes on the interval (0,1). Moreover such $L$-functions have zeroes on the line $\Re(s)=1/2$. Hence, by the definition of $\mathscr{L}(s;q,a)$, we have the following corollary.

\begin{corollary}\label{plattwow}
 For all  pairs $(q, a)$ with  $q \leq 400,000$ and $a$ coprime to $q$, the function $\mathscr{L}(s;q,a)$ has no singularities on the interval $(0,1)$ and has singularities on the line $\Re(s)=1/2$. 
\end{corollary}

Our next result states  that it is possible to resolve Question \ref{GQ} under certain assumptions on the singularities of $\mathscr{L}(s;q,a)$. 

\begin{theorem}\label{BigUncond}
If $\mathscr{L}(s;q,a)$ has a singularity $\rho$ for which $0 < \Re(\rho) < 1$, but does not have singularities on the interval $(0,1)$, then there exists a sequence of $x$ that tends to infinity for which $\log{f(x;q,a)} < 0$. 
\end{theorem}

The following is a direct corollary of Theorem \ref{BigUncond} and Corollary \ref{plattwow}.

\begin{corollary}
Let $q \leq 400,000$. Then there are infinitely many primorials $\overline{N}_k$ in $S_{q,a}$ for which 
\[\frac{\bar{N}_k}{\varphi(\bar{N}_k)(\log(\varphi(q)\log{\bar{N}_k}))^{\frac{1}{\varphi(q)}}} > \frac{1}{C(q,a)}.\]
\end{corollary}

We will write $\Theta:= \Theta (q, a)$ to denote the supremum of the real parts of the singularities of $\mathscr{L}(s;q,a)$ in the strip $0 < \Re(s) < 1$.  By appealing to the functional equations of Dirichlet $L$-functions, it is straightforward to show that the falsehood of SH${}_{q,a}$ is equivalent to $\Theta>1/2$. 
The following theorem provides more precise answers to Question \ref{GQ} under certain assumptions on the singularities of $\mathscr{L}(s;q,a)$.

\begin{theorem}\label{big5}
(a)  Suppose Hypothesis \ref{zeroconj}  (SH${}_{q, a}$) is false. For $1- \Theta < b < \frac{1}{2}$ assume that 
$\mathscr{L}(s;q,a)$ has no singularities on the segment $(1-b,1)$. Then, we have 
 \[\log f(x;q,a) = \Omega_{\pm}(x^{-b}),\]
 as $x$ tends to infinity.\footnote{See Notation \ref{notation} at the end of this section for a definition of the $\Omega$ notation.}
 
(b) For integer $a$ that is not a square modulo $q$, suppose Hypothesis \ref{zeroconj} (SH${}_{q, a}$) is true and $\mathscr{L}(s;q,a)$ does have a singularity on the line $\Re(s)=1/2$ and does not have a singularity at $s=1/2$. Then, for  $\frac{1}{2} < b < \frac{2}{3}$,
 \[\log f(x;q,a) = \Omega_{\pm}(x^{-b}),\]
as $x$ tends to infinity. 

\end{theorem}

Combining the results of part (a) and (b) of the above theorem for $a$ not a square modulo $q$ with Corollary \ref{plattwow} yields the following.

\begin{corollary}\label{wow2} Let $q\leq 400,000$ and let $a$ not be a square modulo $q$. Then there are infinitely many $k \in \N$ for which 
\begin{equation}\label{intronico}
\frac{\bar{N}_k}{\varphi(\bar{N}_k)(\log(\varphi(q)\log{\bar{N}_k}))^{\frac{1}{\varphi(q)}}} > \frac{1}{C(q,a)}
\end{equation}
and also infinitely many $k \in \N$ for which \eqref{intronico} does not hold.
\end{corollary}

Theorem \ref{big5} provides a satisfactory answer to Question \ref{GQ} when $a$ is not a square modulo $q$ and for a square $a$ mod $q$ when SH${}_{q, a}$ is false. It remains to study Question \ref{GQ} when  $a$ is a square modulo $q$ under the assumption of the truth of SH${}_{q, a}$. We study this case by developing an explicit formula for $\log f(x;q,a)$. We start with some notation.

Let $\mathrm{Ind}_q(a)$,  the \emph{index of $a \imod{q}$}, be the least natural number $n >1$ for which $a$ is an $n$-th power modulo $q$
and let \[\mathcal{R}_{q,a} = \#\{b \in {\left(\Z/q\Z\right)^\times}; ~ b^{\mathrm{Ind}_q(a)} \equiv a \imod{q}\}.\] 
Observe that since $q$ and $a$ are coprime, we have $a^{\varphi(q)+1} \equiv a \imod{q}$.  Therefore, $2 \leq \mathrm{Ind}_q(a) \leq \varphi(q)+1$ and thus $\mathrm{Ind}_q(a)$ and $\mathcal{R}_{q,a}$ are well-defined. Also note that $a$ is a square modulo $q$ if and only if $\mathrm{Ind}_q(a)=2$.

With the additional notation 
\[\mathcal{Z}(\chi) = \{ \rho \in \C \text{ ; } L(\rho,\chi) = 0 \text{, } \Re(\rho) \geq 0 \text{ and } \rho \neq 0 \},\]
where as before $\chi'$ denotes the primitive Dirichlet character which induces the Dirichlet character $\chi$, we have the following explicit formula for $\log{f(x;q,a)}$.

\begin{theorem}\label{oscillation}
 Assume Hypothesis \ref{zeroconj} (SH${}_{q, a}$). Writing $m = \mathrm{Ind}_q(a)$, we have
\begin{equation}\label{bigequal}
\log{f(x;q,a)} = \frac{1}{\varphi(q)}\left(\frac{1}{\sqrt{x}\log{x}} \sum_{\chi \imod{q}} \overline{\chi}(a) \sum_{\rho \in \mathcal{Z}(\chi')} \frac{x^{i\Im(\rho)}}{\rho(\rho -1)} - \frac{m\cdot\mathcal{R}_{q,a}}{(m-1)x^{\frac{m-1}{m}}\log{x}}\right)
+ \bigO\left(\frac{1}{\sqrt{x}\log^2{x}}\right).
\end{equation}
\end{theorem}

For $m = 2$, we observe that the negative term in \eqref{bigequal} is of the same order of magnitude as the term corresponding to the sum over zeroes. Furthermore, if the constant $2\mathcal{R}_{q,a}$ is larger than the limit superior (with respect to  $x$) of 
\begin{equation}\label{sumzeroes}
\sum_{\chi \imod{q}} \overline{\chi}(a) \sum_{\rho \in \mathcal{Z}(\chi')} \frac{x^{i\Im(\rho)}}{\rho(\rho -1)},
\end{equation}
then we would have that $\log{f(x;q,a)}\geq 0$ for only finitely many $x$. It is suspected that as $x$ varies, \eqref{sumzeroes} oscillates in sign, so if $2\mathcal{R}_{q,a}$ is \emph{not} larger than the limit superior of \eqref{sumzeroes}, then the sign of $\log{f(x;q,a)}$ will change infinitely often. On the other hand, if $m > 2$, then the negative term has smaller order of magnitude than the order of magnitude of the error term in \eqref{bigequal}. In this case, the oscillatory behavior of \eqref{sumzeroes} will dominate the behavior of $\log{f(x;q,a)}$, so as we proved in part (b) of Theorem \ref{big5} $\log{f(x;q,a)}$ will change sign infinitely often. 

By the above discussion, we have established the following assertion as a corollary of Theorem \ref{oscillation}.

\begin{corollary}\label{chapter6}
Suppose that $a$ is a square modulo $q$ and that Hypothesis \ref{zeroconj} (SH${}_{q, a}$) holds. Then there is a positive number $x_0$ such that 
$$\log{f(x;q,a)}<0$$ 
for $x>x_0$ if and only if 
\begin{equation}
\label{sup}
\limsup_{x \to \infty}{\sum_{\chi \imod{q}} \overline{\chi}(a) \sum_{\rho \in \mathcal{Z}(\chi')} \frac{x^{i\Im(\rho)}}{\rho(\rho -1)} < 2\mathcal{R}_{q,a}}.
\end{equation}
\end{corollary}

Computing the limit superior in Corollary \ref{chapter6} appears to be difficult. Towards an understanding of this limit, we note that the sum over zeroes in Corollary \ref{chapter6} can be bounded in absolute value by
$$
\sum_{\chi \imod{q}} \sum_{\substack{\rho \in \mathcal{Z}(\chi') \\ \Re(\rho) = 1/2}} \frac{1}{\rho(1-\rho)}.
$$
Under the assumption of  $\textrm{GRH}_q$  the above sum is the same as
\begin{equation}
\label{fqdef}\mathcal{F}_q := \sum_{\chi \imod{q}} \sum_{\substack{\rho \in \mathcal{Z}(\chi') }} \frac{1}{\rho(1-\rho)}.
\end{equation}

For several small values of $q$, we have computed $\mathcal{F}_q$ as listed in Table \ref{computable} using a combination of theoretical tools (most notably formulas \eqref{f1} and \eqref{Fcalc}) and Sage \cite{sagemath}.

\begin{table}[h]
\centering
\caption{$\mathcal{F}_q$ for some small values of $q$}
\label{computable}
\begin{tabular}{ll|ll}
\hline
\rowcolor[HTML]{EFEFEF} 
$q$ & \multicolumn{1}{c|}{\cellcolor[HTML]{EFEFEF}$\mathcal{F}_q$} & $q$ & \multicolumn{1}{c}{\cellcolor[HTML]{EFEFEF}$\mathcal{F}_q$} \\ \hline
1   & 0.04619                                                      & 8   & 0.75326                                                     \\
2   & 0.04619                                                      & 9   & 1.41121                                                     \\
3   & 0.15942                                                      & 10  & 0.60919                                                     \\
4   & 0.20176                                                      & 11  & 4.26098                                                     \\
5   & 0.60919                                                      & 12  & 0.64516                                                     \\
6   & 0.15942                                                      & 13  & 6.45484                                                     \\
7   & 1.41418                                                      & 14  & 1.41418                                                    
\end{tabular}
\end{table}

Note that for $q\neq 11, 13$ and a square $a$ mod $q$, none of the values of $\mathcal{F}_q$ in the above table are larger than the smallest possible value of $2\mathcal{R}_{q,a}$, which is 4 for $q > 2$. (We have also calculated the values of $\mathcal{F}_q$ for $q = p, 2p$ for primes $p \leq 149$. However, once $p > 7$, we found that $\mathcal{F}_p = \mathcal{F}_{2p}$ is larger than $4=2\mathcal{R}_{q, a}$ for square $a$ mod $q$.) 
Therefore, we have the following proposition.

\begin{proposition}\label{listcor}
Let $q \leq 10$ or $q= 12, 14$ and assume $\textrm{GRH}_q$. We have 
\[\limsup_{x \to \infty}{\sum_{\chi \imod{q}} \overline{\chi}(a) \sum_{\rho \in \mathcal{Z}(\chi')} \frac{x^{i\Im(\rho)}}{\rho(\rho -1)}\leq \mathcal{F}_q < 2\mathcal{R}_{q,a}}.\] 
\end{proposition}

Combining the results of Proposition \ref{listcor} and Corollary \ref{chapter6}  we arrive at the following connection between $\textrm{GRH}_q$ and inequality \eqref{GeneralIneq}.

\begin{proposition}\label{Conc1}
Let $q \leq 10$ or $q=12, 14$ and let $a$ be a square modulo $q$. Assuming $\textrm{GRH}_q$, there are at most finitely $k \in \N$ for which 
\[\frac{\bar{N}_k}{\varphi(\bar{N}_k)(\log(\varphi(q)\log{\bar{N}_k}))^{\frac{1}{\varphi(q)}}} \leq \frac{1}{C(q,a)}.\]
\end{proposition}

Combining the results of Proposition \ref{Conc1} with  part (a) of Theorem \ref{big5} and by employing Corollary \ref{plattwow} and some numerical computations we deduce several generalizations of Criterion \ref{NicolasRH}.

\begin{theorem}\label{criterion}
For $q\leq 10$ and for $q=12, 14$, $\textrm{GRH}_q$ is true if and only if 
for all positive integers $k$ we have
\[\frac{\bar{N}_k}{\varphi(\bar{N}_k)(\log(\varphi(q)\log{\bar{N}_k}))^{\frac{1}{\varphi(q)}}} > \frac{1}{C(q,1)}.\]
\end{theorem}

\begin{remarks}\label{rem}
(i) By using a theorem of Dirichlet (\cite[Theorem 201, p. 218]{hardy2008}) on simultaneous approximation of real numbers by rationals one can show 
$$\liminf_{x \to \infty}{\sum_{\chi \imod{q}}}  \sum_{\rho \in \mathcal{Z}(\chi')} \frac{x^{i\Im(\rho)}}{\rho(\rho -1)} =\sum_{\chi \imod{q}}  \sum_{\rho \in \mathcal{Z}(\chi')} \frac{1}{\rho(\rho -1)}=-\mathcal{F}_q.$$
(ii) Following an argument analogous to \cite[Theorem 33]{ingham1932} we can show that 
$$\limsup_{x \to \infty}{\sum_{\chi \imod{q}}}  \sum_{\rho \in \mathcal{Z}(\chi')} \frac{x^{i\Im(\rho)}}{\rho(\rho -1)}>\frac{{\rm Res}_{s=\rho_1}\mathscr{L}(s;q,1)}{|\rho_1 (1-\rho_1)|},$$
where $\rho_1$ is the first singularity (the singularity with the lowest ordinate)  of $\mathscr{L}(s;q,1)$ in the critical strip. 
 (iii) By employing \eqref{f1} and \eqref{Fcalc} we have (for $q>2$)
 \begin{equation}
 \label{fq}
 \mathcal{F}_q=\sum_{\substack{d\mid q\\ d\neq 1}} \varphi^*(d) \log\frac{d}{\pi}+2\sum_{\substack{d\mid q \\d\neq 1}} \sum_{\chi \starmod{d}} \frac{L^\prime}{L}(1, \chi)-\varphi(q)(\gamma+\log{2})+2\gamma-\log{\pi}+2,
 \end{equation}
 where 
 $\varphi^*(d)$ is the number of primitive characters mod $d$ and $\chi \starmod{d}$ denotes a primitive Dirichlet character mod $d$. From \cite[Theorem 1.4]{PR} we know that
 $$\frac{1}{\varphi^*(d)} \sum_{\chi \starmod{d}} \left| \frac{L^\prime}{L} (1, \chi) \right|^2=\sum_{n=1}^{\infty} \frac{\Lambda(n)^2}{n^2}- \sum_{p\mid d} \frac{\log^2{p}}{h(p, d)}+O(d^{-1/10}),$$
where $\Lambda(n)$ is the von Mangoldt function, $h(p, d)=(p-1)^2$ when $p^2\mid d$ and $h(p, d)=p^2-1$ otherwise. By applying the Cauchy-Schwarz inequality in the term involving $\frac{L'}{L}(1,\chi)$ in \eqref{fq} and employing the above identity for  $\sum_{\chi \starmod{d}} \left| \frac{L^\prime}{L} (1, \chi) \right|^2$, we conclude that 
\begin{equation}
\label{previous}
\mathcal{F}_q=\sum_{d\mid q} \varphi^*(d) \log\frac{d}{\pi} +O(\varphi(q)).
\end{equation}
Since $\displaystyle{\phi^*(q)=q\prod_{p\| q} (1-\tfrac{2}{p})\prod_{p^2\mid q} (1-\tfrac{1}{p})^2}$ (see \cite[p. 286]{montgomery2007}) and $\mathcal{R}_{q,1}=O(2^{\omega(q)})$ (see Proposition \ref{rqa}), where $\omega(q)$ is the number of prime divisors of $q$, from \eqref{previous} we conclude that 
$$\lim_{q\rightarrow \infty} {\mathcal{F}_q}/{\mathcal{R}_{q,1}}=\infty.$$
Thus, one can ask for determination of the finite set of integers $q$ for which $\mathcal{F}_q < 2\mathcal{R}_{q, 1}.$
\end{remarks}

In view of the above discussion and remarks it would be interesting to investigate the following.

\begin{question}\label{question}
Is it true that 
$$\limsup_{x \to \infty}{\sum_{\chi \imod{q}}}  \sum_{\rho \in \mathcal{Z}(\chi')} \frac{x^{i\Im(\rho)}}{\rho(\rho -1)}=\mathcal{F}_q?$$
\end{question}
By part (iii) of Remarks \ref{rem}, a positive answer to Question \ref{question} implies that a Nicolas type criterion for $\textrm{GRH}_q$ (similar to the one given in Theorem \ref{criterion}) can be established only for finitely many values of $q$. 

\medskip\par

The structure of this paper is as follows. First, we will ensure that the questions we are asking are justified by proving Theorem \ref{app}. From there, in Section \ref{Useful Expressions}
we will establish several useful estimates for $\log{f(x;q,a)}$ to be used throughout the paper. Once these estimates are in place,  in Sections \ref{GRH False} and \ref{sec:Omega Theoremsf}, we will turn our attention to establishing Theorem  \ref{BigUncond} and Theorem \ref{big5}. In Section \ref{GRH True} we
 prove Theorem \ref{oscillation}, which is a key tool for examining the behavior of $\log{f(x;q,a)}$.  Section \ref{comp} is dedicated to computation of several values of $\mathcal{F}_q$. Finally we prove Theorem \ref{criterion} in Section \ref{eight}.

\medskip\par

\begin{notation}\label{notation}
Throughout this paper, $\phi(n)$ is Euler's totient function and $\gamma$ is always the Euler-Mascheroni constant. The numbers $q$ and $a$ will be fixed positive integers, usually coprime. For a pair of coprime $q$ and $a$ we have the set $S_{q,a}= \left\{n \in \N \mbox{ ; } p \mid n \implies p \equiv a \imod{q}\right\}$, which includes every $k$-th prime $\bar{p}_k$ in the progression $p \equiv a \imod{q}$ and also every $k$-th primorial $\bar{N}_k =  \prod_{i = 1}^k \bar{p}_k$ arising from this progression. We follow the usual conventions of analytic number theory with respect to asymptotic notations, with the inclusion of the less common $\Omega$ notation. For one, $f(x) = \Omega_{+}(g(x))$ if there exists a positive constant c and an increasing real sequence which tends to infinity along which $f(x) > c g(x)$. Likewise $f(x) = \Omega_{-}(g(x))$ if there exists a positive constant c and an increasing real sequence which tends to infinity along which $f(x) < - c g(x)$. If both $f(x) = \Omega_{+}(g(x))$ and $f(x) = \Omega_{-}(g(x))$, we write $f(x) = \Omega_{\pm}(g(x))$. We use $\left(\Z/q\Z\right)^\times$ to denote the multiplicative group of integers modulo $q$. The real and imaginary parts of a complex number $\rho$ are denoted by $\Re(\rho)$ and $\Im(\rho)$, respectively.
\end{notation}

\section{Proof of Theorem \ref{app}}

\begin{proof}[Proof of Theorem \ref{app}]
The proof is an adaptation of the proof of {\cite{hardy2008}*{Theorem 328}} to the case of integers in $S_{q, a}$.
 For $n \in S_{q, a}$, let $r$ be the number of prime divisors of $n$ that are larger than $\varphi(q)\log{n}$. Writing $n = p_1^{\alpha_1} p_2^{\alpha_2} \ldots p_k^{\alpha_k}$, we have $(\varphi(q)\log{n})^r < n$ and thus,
\[r < \frac{\log{n}}{\log(\varphi(q)\log{n})}.\]
Employing the above bound for $r$ yields

\begin{equation}\label{star}
\begin{aligned}[b]
 \frac{n}{\varphi(n)} = \prod_{i = 1}^k \left(1 - \frac{1}{p_i}\right)^{-1} &\leq (1 - \frac{1}{\varphi(q)\log{n}})^{-r} \prod_{\substack{p \leq \varphi(q)\log{n} \\ p \mid n}} \left(1 - \frac{1}{p}\right)^{-1} \\
&< \left(1 - \frac{1}{\varphi(q)\log{n}}\right)^{\frac{-\log{n}}{\log(\varphi(q)\log{n})}} \prod_{\substack{p \leq \varphi(q)\log{n} \\ p \equiv a \imod{q}}} \left(1 - \frac{1}{p}\right)^{-1}. 
\end{aligned}
\end{equation}
The first factor on the right of \eqref{star} tends to 1 as $n \rightarrow \infty$. By invoking Theorem \ref{mertensap} for the latter product, we conclude that

\begin{equation}\label{starstar}
  \left(1 - \frac{1}{\varphi(q)\log{n}}\right)^{\frac{-\log{n}}{\log(\varphi(q)\log{n})}} \prod_{\substack{p \leq \varphi(q)\log{n} \\ p \equiv a \imod{q}}} \left(1 - \frac{1}{p}\right)^{-1} 
\sim \frac{(\log(\varphi(q)\log{n}))^\frac{1}{\varphi(n)}}{C(q, a)},
\end{equation}
as $n \rightarrow \infty$. From \eqref{star} and \eqref{starstar}, we deduce

\[\limsup_{n \in S_{q, a}} \frac{n}{\varphi(n)(\log(\varphi(q)\log{n}))^{1/\varphi(q)}} \leq \frac{1}{C(q,a)}. \]

To establish a sequence which attains this bound, consider $\bar{N}_k$, the $k$-th primorial in $S_{q,a}$. 
Then, by Theorem \ref{mertensap} 

\[\frac{\bar N_k}{\varphi(\bar N_k)} = \prod_{\substack{p \leq \bar p_k \\ p \equiv a \imod{q}}} \left(1 - \frac{1}{p}\right)^{-1} \sim \frac{(\log{\bar p_k})^\frac{1}{\varphi(q)}}{C(q,a)}, \]
as $k \rightarrow \infty$. Next, we apply \eqref{pntap} to obtain

\[\log(\varphi(q)\log{\bar N_k}) = \log(\varphi(q){\theta(\bar p_k ; a, q)}) \sim \log{\bar p_k},\]
as $k \rightarrow \infty$. 
Hence, we have

\[\frac{\bar N_k}{\varphi(\bar N_k)} \sim \frac{\left(\log(\varphi(q)\log{\bar N_k})\right)^{\frac{1}{\varphi(q)}}}{C(q,a)},\]
as $k \rightarrow \infty$. That is,
\[\lim_{k \rightarrow \infty} \frac{\bar N_k}{\varphi(\bar N_k)(\log(\varphi(q)\log{\bar N_k}))^{1/\varphi(q)}} = \frac{1}{C(q,a)}.\]
This concludes the proof. 
\end{proof}

\section{Useful Expressions for $\ensuremath{\log f(x;q,a)}$}\label{Useful Expressions}

To establish an initial expression for $\log{f(x;q,a)}$ in terms of prime counting functions, it is beneficial to develop a variety of estimates for some related functions. First, let
\[g(x) := -\diff[2]{}{x}\left(\log\log{x}\right) = \frac{1+\log{x}}{x^2\log^2{x}}.\]

Second, for a given arithmetic progression consider the error term in the prime number theorem, which will be denoted
\[S(x;q,a) := \theta(x;q,a) - \frac{x}{\varphi(q)}.\]

We next obtain an identity for $\log{f(x;q, a)}$ with respect to
\[K(x;q,a) := \int_x^\infty \! S(t;q,a) g(t) \,dt.\]

\begin{proposition}\label{logup}
Let $q, a \in \N$ be fixed coprime integers. Then, as $x\rightarrow \infty$,
 \begin{equation}\label{inceptionfinal}
\log{f(x;q,a)} = K(x;q,a) + \bigO\left(\tfrac{1}{x}\right).
\end{equation}
\end{proposition}

\begin{proof}

By partial summation
\[
   \begin{aligned}[b]
     \sum_{\substack{p \leq x \\ p \equiv a\imod{q}}} \frac{1}{p} 
     &= \frac{\theta(x;q,a)}{x\log{x}} + \int_{\bar{p}_1}^x \! \theta(t;q,a)g(t) \, dt.
   \end{aligned}
\]
From here with the substitution $\theta(t;q,a) = S(t;q,a) + \tfrac{t}{\varphi(q)}$, we obtain
\begin{equation}\label{mertensadd}
 \begin{aligned}
   \sum_{\substack{p \leq x \\ p \equiv a\imod{q}}} \frac{1}{p} &= \frac{S(x;q,a)}{x\log{x}} + \frac{1}{\varphi(q)\log{x}} + \int_{\bar{p}_1}^x \! S(t;q,a)g(t) \, dt + \frac{1}{\varphi(q)}\int_{\bar{p}_1}^x \! \frac{1}{t}\frac{\log{t}+1}{\log^2{t}} \, dt. 
 \end{aligned}
\end{equation}
Hence, we may write \eqref{mertensadd} as
\begin{equation}\label{psf}
 \sum_{\substack{p \leq x \\ p \equiv a\imod{q}}} \frac{1}{p} = \frac{S(x;q,a)}{x\log{x}} + \frac{\log\log{x}}{\varphi(q)} - K(x;q,a) + M(q,a),
\end{equation}
where
\[M(q,a) = \int_{\bar{p}_1}^\infty \! S(t;q,a)g(t) \, dt + \frac{1}{\varphi(q)\log{\bar{p}_1}} -\frac{\log\log{\bar{p}_1}}{\varphi(q)}\]
is the constant term. By comparison with Mertens' second theorem for arithmetic progressions (\cite[(1-1)]{languasco2010}), it can be shown that this constant has the expression 
\begin{equation}\label{Mexpr}
 M(q,a) = \sum_{p \equiv a\imod{q}} \left\{ \log{\left(1-\frac{1}{p}\right)} + \frac{1}{p} \right\} - \log{C(q,a)},
\end{equation}
(see \cite[(1-3)]{languasco2010}).
Now it can be readily verified, by \eqref{equivalence} and \eqref{Mexpr}, that 
\begin{equation}\label{fuu}
\begin{aligned}[b]
\varphi(q)\log{f(x;q,a)} &= \log\log(\varphi(q)\theta(x;q,a)) + \varphi(q)\sum_{\substack{p \leq x \\ p \equiv a\imod{q}}} \log{\left(1-\frac{1}{p}\right)} -\varphi(q)\log{C(q,a)} \\&=  \bar{U}(x) + \underline{u}(x),
\end{aligned} 
\end{equation} 
where
\begin{equation}\label{bigu} \bar{U}(x) = \log\log(\varphi(q)\theta(x;q,a)) - \varphi(q)\sum_{\substack{p \leq x \\ p \equiv a\imod{q}}} \frac{1}{p} + \varphi(q)M(q,a)\end{equation}
and
\begin{equation}\label{littleu} \begin{aligned}[b]
    \underline{u}(x) &= 
-\varphi(q)\sum_{\substack{p > x \\ p \equiv a\imod{q}}}  \left\{ \log{\left(1-\frac{1}{p}\right)}  + \frac{1}{p} \right\}.
   \end{aligned}
\end{equation}
By crudely bounding \eqref{littleu} with a geometric series, we see that 
\begin{equation}\label{geou}
0 < \underline{u}(x) \leq \frac{\varphi(q)}{2(x-1)}.
\end{equation} 

Now from \eqref{fuu}, \eqref{bigu}, and \eqref{geou}, we have for $x \geq \bar{p}_1$ that
\[\varphi(q)\log{f(x;q,a)} = \log\log(\varphi(q)\theta(x;q,a)) -\varphi(q) \sum_{\substack{p \leq x \\ p \equiv a\imod{q}}} \frac{1}{p} + \varphi(q)M(q,a) + \underline{u}(x).\]
Substituting equation \eqref{psf} for the series in the above equation yields
\begin{equation}\label{brackee}
   \varphi(q)\log{f(x;q,a)} = \log\log(\varphi(q)\theta(x;q,a)) - \frac{\varphi(q)S(x;q,a)}{x\log{x}} - \log\log{x} + \varphi(q)K(x;q,a) + \underline{u}(x).
\end{equation}

By the mean value theorem for $h(t)=\log\log{t}$, there exists a number $c$ between $x$ and $\varphi(q)\theta(x;q,a)$ for which
\begin{equation*}\label{tcc}
\log{\log(\varphi(q)\theta(x;q,a))} = \log\log{x} + \frac{\varphi(q)S(x;q,a)}{c\log{c}}.
\end{equation*}
From here, we arrive at
\begin{equation}\label{equal2}
\log\log(\varphi(q)\theta(x;q,a)) - \log\log{x} - \frac{\varphi(q)S(x;q,a)}{x\log{x}} = \varphi(q)S(x;q,a)\left(\frac{x\log{x} - c\log{c}}{(x\log{x})(c\log{c})}\right).
\end{equation}
Combining \eqref{brackee} and \eqref{equal2}, we have the identity
\begin{equation}\label{inception}
\log{f(x;q,a)} = K(x;q,a) + S(x;q,a)\left(\frac{x\log{x} - c\log{c}}{(x\log{x})(c\log{c})}\right) +\frac{\underline{u}(x)}{\phi(q)}.
\end{equation}

Consider the second term in the right hand side of  \eqref{inception} and assume that $x<\varphi(q)\theta(x;q,a)$ . We see that with $\epsilon>0$ chosen so that $x < c < \varphi(q)\theta(x;q,a) \leq (1+\epsilon)x$, we have 

\begin{equation}
\label{mvt}
  \left| \frac{1}{x\log{x}} - \frac{1}{c\log{c}} \right| \leq \left| \frac{1}{x\log{x}} - \frac{1}{(1+\epsilon)x\log{(1+\epsilon)x}} \right| = \left| \frac{\log{\xi} + 1}{(\xi \log{\xi})^2} \right| \ll \frac{1}{x^2 \log{x}}
\end{equation}
for $x < \xi <(1+\epsilon)x$ arising from an application of the mean value theorem for the function $s(t)=1/t\log{t}$. Using the upper bound $S(x;q,a) \ll \tfrac{x}{\log{x}}$ together with \eqref{mvt}, we  have
\[
 \left|S(x;q,a)\left(\frac{x\log{x} - c\log{c}}{(x\log{x})(c\log{c})}\right)\right| \ll  \frac{1}{x \log^2{x}}.
\]
A similar bound holds if $\varphi(q)\log\log{x}<x$.

Therefore, recalling \eqref{geou}, \eqref{inception} becomes
\[\log{f(x;q,a)} = K(x;q,a) + \bigO\left(\tfrac{1}{x}\right),\]
as desired. \end{proof}
Next, we shift our attention from $\theta(x;q,a)$ to 
\[\psi(x;q,a) = \sum_{\substack{p^k \leq x \\ p^k \equiv a \imod{q}}} \log{p},\]
for which we have analogues
\begin{equation}
\label{Rdefinition}
R(x;q,a) := \psi(x;q,a) - \frac{x}{\varphi(q)}
\end{equation}
and
\[J(x;q,a) := \int_x^\infty \! R(t;q,a) g(t) \,dt.\]

By definition, $\theta(x;q,a) \leq \psi(x;q,a)$, and therefore $K(x;q,a) \leq J(x;q,a)$ when $x>e^{-1}$. In order to study the precise relation between $K(x;q, a)$ and $J(x;q,a)$ we need some notations. Recall that we denoted by $\mathrm{Ind}_q(a)$, the \emph{index of $a \imod{q}$}, the least natural number $n >1$ for which $a$ is an $n$-th power modulo $q$. Furthermore, we set
\[\mathcal{R}_{q,a} = \#\{b \in {\left(\Z/q\Z\right)^\times};~  b^{\mathrm{Ind}_q(a)} \equiv a \imod{q}\}.\]
It will be useful to have a closed form for $\mathcal{R}_{q,a}$. 

\begin{proposition}\label{rqa}
 Write $q = 2^\alpha q_1^{\alpha_1} q_2^{\alpha_2} \ldots q_r^{\alpha_r}$, where $q_i$ are the distinct odd prime divisors of $q$. Let $m = \mathrm{Ind}_q(a)$. We have
 \[\mathcal{R}_{q,a} = \begin{cases} 
 \prod_{i=1}^r \left(m, \varphi(q_i^{\alpha_i})\right) &\mbox{if } \alpha \leq 1, \\ 
               \left(m,2\right)\left(m, 2^{\alpha-2}\right) \prod_{i=1}^r \left(m, \varphi(q_i^{\alpha_i})\right) & \mbox{otherwise. } 
               \end{cases}\]
\end{proposition}

\begin{proof}
By Theorem 3.21 of \cite{leveque1996}, we know that the solutions of $x^m \equiv a \imod{q}$ are in 1-1 correspondence with the solutions of the system
 \[
\left\{ 
\begin{array}{c}
x^m \equiv a \imod{2^\alpha}, \\ 
 x^m \equiv a \imod{q_1^{\alpha_1}}, \\ 
 \vdots \\
 x^m \equiv a \imod{q_r^{\alpha_r}}.
\end{array}
\right.
\]
 For each odd prime $q_i$, Theorem 4.13 of \cite{leveque1996} establishes that there are $\left(m, \varphi(q_i^{\alpha_i})\right)$ solutions to each congruence $x^m \equiv a \imod{q_i^{\alpha_i}}$.
 On the other hand, the congruence $x^m \equiv a \imod{2^{\alpha}}$ has 1 solution if $\alpha = 1$, again by \cite[Theorem 4.13]{leveque1996}. If $\alpha \geq 2$, then $x^m \equiv a \imod{2^{\alpha}}$ has $(m, 2)\cdot(m, 2^{\alpha-2})$ solutions via Theorem 4.14 of \cite{leveque1996}. The formula for $\mathcal{R}_{q,a}$ follows by taking the product of the number of solutions as we range over congruences corresponding to prime divisors of $q$. 
\end{proof}

\begin{proposition}\label{transfer}
 Writing $m = \mathrm{Ind}_q(a)$, we have
 \[\theta(x;q,a) = \psi(x;q,a) -\frac{\mathcal{R}_{q,a}}{\varphi(q)}{x^{\frac{1}{m}}} + \bigO\left(\frac{x^{\frac{1}{m}}}{\log{x}}\right).\]
\end{proposition}
\begin{proof}
Consider
\begin{equation}\label{breaking}
\psi(x;q,a) - \theta(x;q,a) = \sum_{k=2}^\infty \sum_{b \in ({\Z}/{q\Z})^\times} \sum_{\substack{p^k \leq x \\ p^k \equiv a\imod{q} \\ p \equiv b\imod{q}}} \log{p}. 
\end{equation}
Equivalently,
\[\theta(x;q,a) = \psi(x;q,a) - \sum_{k=2}^\infty \sum_{b \in ({\Z}/{q\Z})^\times} \eta_b(k) \sum_{\substack{p \leq x^{\frac{1}{k}} \\ p \equiv b\imod{q}}} \log{p},\]
where 
\[\eta_b(k) = \begin{cases}
1 & \textnormal{if } b^k \equiv a \imod{q}, 
\\
0 &\textnormal{otherwise. }
\end{cases}\]
Observe that $\eta_b(k) = 0$ for all $k < m=\mathrm{Ind}_q(a)$. Thus, we have
\begin{equation}\label{tetasi}
\theta(x;q,a) = \psi(x;q,a) - \sum_{b \in ({\Z}/{q\Z})^\times}\eta_b(m)\theta(x^{\frac{1}{m}};q,b) - \sum_{k = m+1}^\infty \sum_{b \in ({\Z}/{q\Z})^\times} \eta_b(k) \theta(x^{\frac{1}{k}};q,b).
\end{equation}

By \cite[Theorem 6.8]{narkiewicz2000}, we know that, for a positive constant $\beta_	q$, depending only on $q$,
\[\theta(x;q,b) = \frac{x}{\varphi(q)} + \bigO(x \exp(-\log^{\beta_q} x))\]
and therefore,
\begin{equation}
\label{pnt}
\sum_{b \in ({\Z}/{q\Z})^\times} \eta_b(k)\theta(x^{\frac{1}{k}};q,b) = \frac{\mathcal{R}_{q,a}}{\varphi(q)}{x^{\frac{1}{k}}} + \bigO(x^{\frac{1}{k}}\exp(-\log^{\beta_q}{(x^{\frac{1}{k}})})).
\end{equation}

Applying \eqref{pnt} in \eqref{tetasi} yields the result.
\end{proof}
In order to apply Proposition \ref{transfer} in an expression for $J(x;q, a)$ we need to integrate a version of the identity of Proposition \ref{transfer} weighted with $g(x)$ introduced at the beginning of this section. In this direction we consider
\[F_s (x) = \int_x^\infty \! t^{s}g(t) \,dt.\]

The following lemma is due to Nicolas (\cite[Lemma 2.2]{nicolas2012}).

\begin{lemma}\label{nicolaslem}
 Let $s$ be a complex number such that $\Re(s) < 1$. Then, for $x >1$,
\[F_s(x) = -\frac{x^{s-1}}{(s-1)\log{x}} + r_s(x),\]
where 
\[r_s(x) = -\frac{s}{1-s}\left( \frac{x^{s-1}}{(1-s)\log^2{x}} + \int_x^\infty \! \frac{2t^{s-2}}{(s-1)\log^3{t}} \, dt\right).\]
\end{lemma}
As a direct consequence of the above expression for $r_s(x)$ we have
\begin{equation}
\label{rs}
\begin{aligned}
 \lvert r_s(x) \rvert 
&\leq \left\lvert \frac{s}{(1-s)^2} \right\rvert \left(\frac{x^{\Re(s)-1}}{\log^2{x}} \right)\left(1+\frac{2}{\lvert \Re(s)-1 \rvert \log{x}} \right).
\end{aligned}
\end{equation}

Now, by combining Proposition \ref{logup}, Proposition \ref{transfer}, Lemma \ref{nicolaslem}, and \eqref{rs} we have the following expression for $\log{f(x;q,a)}$ in terms of $J(x;q, a)$. 
\begin{proposition}\label{transferup}
 Let $q,a \in \N$ be coprime and $m = \mathrm{Ind}_q(a)$. Then, 
\[\log{f(x;q,a)} =  J(x;q,a) - \frac{\mathcal{R}_{q,a}}{\varphi(q)} \frac{m}{(m-1)x^{\frac{m-1}{m}}\log{x}} + \bigO\left(\frac{1}{x^{\frac{m-1}{m}}\log^2{x}}\right).\]
\end{proposition}

\section{$\Omega$-Theorems for $J(x; q, a)$}\label{GRH False}

We adapt the techniques of \cite[Chapter V]{ingham1932} to establish several $\Omega$-theorems for $J(x;q,a)$. The following classical theorem plays a fundamental role in our arguments.

\begin{theorem}[\textbf{Landau's Oscillation Theorem}]\label{landauoscillation}
 Let $h: \left[ 1, \infty\right) \rightarrow \R$ be a function which is bounded and Riemann-integrable on intervals of the form $[1,T]$, $1<T<\infty$. Consider the integral
\[H(s) = \int_1^\infty \! \frac{h(x)}{x^s} \, \textrm{d}x.\]
Suppose that the line $\Re(s) = \sigma_0$ is the line of convergence for $H$, and the function $h(x)$ is of constant sign on an interval of the form $\left[ x', \infty \right)$. Then the real point $s=\sigma_0$ on the line of convergence must be a singularity of $H(s)$. 
\end{theorem}
\begin{proof}
See \cite[Theorem H, p. 88]{ingham1932}.
\end{proof}

Under some relatively mild conditions we prove that $J(x;q,a)$ oscillates.

\begin{theorem}\label{omegaJ}
 If $\mathscr{L}(s;q,a)$ has no singularities on $(0, 1)$ and it 
has a singularity $\rho$ with $0 < \Re(\rho) < 1$, then we have 
$J(x;q,a) < 0$
for arbitrarily large $x$ and also $J(x;q, a)>0$ for arbitrary large $x$. 
\end{theorem}

\begin{proof}
We start by finding an expression for an integral involving the error term $R(x; q, a)$ defined in \eqref{Rdefinition} in terms of the logarithmic derivatives of Dirichlet $L$-functions. By way of Exercise 2.1.5 of \cite{murty2001} and formula (4.28) of \cite{montgomery2007} we have, for $\Re(s)>1$, 
\[  \int_{\overline{p}_1}^\infty \! \frac{\psi(t;q,a)}{t^{s+1}} \, dt=\frac{1}{s}\sum_{n\equiv a \imod{q}} \frac{\Lambda(n)}{n^s}=  - \frac{1}{s\varphi(q)} \sum_{\chi \imod{q}} \bar{\chi}(a) \frac{L'}{L}(s, \chi).\]
Using the above identity and definition $R(t;q,a) = \psi(t;q,a) - t/\varphi(q)$, we have, for $\Re(s)>1$, 

\begin{equation}\label{phase1} 
 \int_{\bar{p}_1}^\infty \! \frac{R(t;q,a)}{t^{s+1}} \, dt = \frac{1}{\varphi(q)}\left( -\frac{1}{s} \mathscr{L}(s;q,a) - \frac{1}{s-1} + \int_{1}^{\bar{p}_1} \! \tfrac{1}{t^s} \, dt\right).
\end{equation}

Next let $0<\delta<\Im(\rho_1)$, where $\rho_1$ is a singularity of $\mathscr{L}(s;q,a)$ with the smallest positive ordinate. Note that under the stated conditions on singularities of $\mathscr{L}(s;q,a)$, $\rho_1$ and $\delta$ are well-defined. Now let 
\[\displaystyle W_\delta = \{\,s \mbox{ ; } \Re(s) >1\,\} \cup \{\,s \mbox{ ; } 0 < \Re(s) \leq 1\text{ and } \lvert \Im(s) \rvert < \delta\,\}.\]

Observe that the right-hand side of identity \eqref{phase1} is holomorphic on $W_\delta$ (the simple pole at $s=1$ of $\mathscr{L}(s;q,a)$ cancels the simple pole at $s=1$ of $1/(s-1)$). Hence, on $W_\delta$, the right hand side of \eqref{phase1} admits an antiderivative, call it $H_1(s)$, and a second antiderivative $H_2(s)$.

For $\Re(s)>1$, from \eqref{phase1} we have \[\frac{d}{ds} \left( H_1(s) \right) = \int_{\bar{p}_1}^\infty \! \frac{R(t;q,a)}{t^{s+1}} \, dt=\frac{d}{ds}\left(-\int_{\bar{p}_1}^\infty \! \frac{R(t;q,a)}{t^{s+1}\log{t}} \, dt \right).\]
By integrating two sides of the above identity along smooth curves in the half-plane $\Re(s)>1$ and with a fixed initial point we get, for $\Re(s)>1$,
\begin{equation}\label{r1}
\int_{\bar{p}_1}^\infty \! \frac{R(t;q,a)}{t^{s+1}\log{t}} \, dt = -H_1(s) + (\lambda_2 - \lambda_1),
\end{equation}
where $\lambda_1, \lambda_2$ are fixed complex constants.

Recalling that the antiderivative of $H_1$ is $H_2$, we integrate along a smooth curve once more to obtain, for $\Re(s)>1$, 
\begin{equation}\label{r2}
\int_{\bar{p}_1}^\infty \! \frac{R(t;q,a)}{t^{s+1}\log^2{t}} \, dt = H_2(s) + (\lambda_2 - \lambda_1)s + \mu,
\end{equation}
where $\mu$ is a fixed complex constant.

Now for $\Re(s)>1$ set
\begin{equation}\label{h1}
 H(s) {:=} \int_{\bar{p}_1}^\infty \! \frac{J(x;q,a)}{x^s} \, dx  = \int_{\bar{p}_1}^\infty \! \frac{1}{x^s} \, \int_x^\infty \! \frac{R(t;q,a)}{t^2}\frac{\left(\log{t} + 1\right)}{\log^2{t}} \, dt \, dx.
\end{equation}

\begin{figure}
\centering
\begin{tikzpicture}
    \begin{scope}[thick,font=\scriptsize]
    \draw [->] (-0.5,0) -- (5.5,0) node [above left]  {$\Re(s)$};
    \draw [->] (0,-3) -- (0,3) node [below right] {$\Im(s)$};

    \draw (2,-3pt) -- (2,3pt)   node [above] {$1$};
    \draw (1,2.5) -- (1,2.5)   node [below] {$\rho$};
    \draw (0,-3pt) -- (0,3pt) node [above left] {$0$};
    \draw (-3pt,2) -- (3pt,2)   node [right] {$\delta$};
    \draw (-3pt,-2) -- (3pt,-2) node [right] {$-\delta$};

    \end{scope}
    \path [draw=none,fill=black] (1, 2.5) circle (1pt);
    
    \path [draw=none,fill=gray,semitransparent] (0,-2) rectangle (2,2);
    \path [draw=none,fill=gray,semitransparent] (2,-3) rectangle (5.5,3);
    \node [above right,gray] at (+2,+1) {$W_\delta$};
\end{tikzpicture}
\caption{The region $W_\delta$ and the singularity $\rho$.} \label{fig:W}
\end{figure}
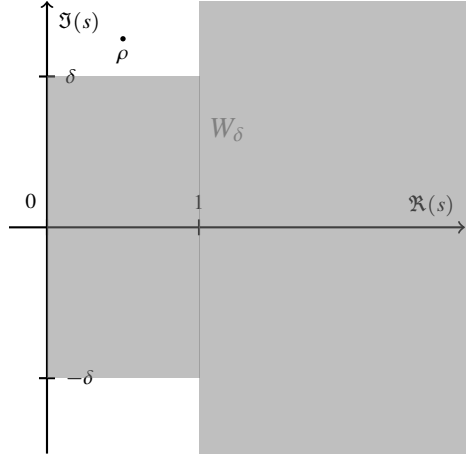

By changing the order of integration in \eqref{h1}, we arrive at  
\[H(s) = \frac{1}{s-1}\left((\bar{p}_1)^{1-s}J(\bar{p}_1;q,a) - \int_{\bar{p}_1}^\infty \! \frac{R(t;q,a)}{t^{s+1}\log{t}} \, dt - \int_{\bar{p}_1}^\infty \! \frac{R(t;q,a)}{t^{s+1}\log^2{t}} \, dt \right).\]
We substitute the integrals in the above equation using their respective identities \eqref{r1} and \eqref{r2} to get 
\begin{equation}\label{H2} 
H(s) =\frac{1}{s-1}\left(H_1(s) - H_2(s) + E(s)\right),
\end{equation}
where 
\[E(s) {:=} (\bar{p}_1)^{1-s}J(\bar{p}_1;q,a) - (\lambda_2 - \lambda_1)s - \mu - (\lambda_2 - \lambda_1).\]
Observe that $E(s)$ is entire while $H_1(s)$ and $H_2(s)$ are holomorphic on $W_\delta$.  Moreover  $H_1(1)-H_2(1)+E(1)=0$. Thus \eqref{H2} establishes an analytic continuation of $H(s)$ to $W_\delta$. 
Thus, crucially, we have extended $H(s)$ to the real line in the critical strip. Hence, if we suppose $J(x;q,a)$ is of constant sign for some interval $\left[ x', \infty \right)$, then Theorem \ref{landauoscillation} establishes that the abscissa of convergence of $H(s)$ must satisfy $\Re(s) \leq 0$, since no point with $s = \sigma > 0$ is a singularity. That is, $H(s)$ must extend to a function which is holomorphic in the half-plane $\Re(s) >0$.

Reconsidering \eqref{H2}, we see that the holomorphy of $H(s)$ implies that $H_1(s) - H_2(s)$ is holomorphic on $\Re(s) >0$, and therefore $\frac{d^2}{ds^2}\left(H_1(s) - H_2(s)\right)$ is holomorphic in this region as well. We have assumed $\mathscr{L}(s;q,a)$ has a singularity at $s=\rho$, where $0< \Re(\rho) < 1$ and $\lvert\Im(\rho)\rvert > 0$. Such a singularity must be simple, since the zeroes of $L(s,\chi)$ contribute simple poles with residue $m_\rho(\chi)$ in the logarithmic derivative. Therefore, in an appropriate deleted neighborhood of $\rho$, 
\[\mathscr{L}(s;q,a) = \frac{\sum_{\chi \imod{q}} \overline{\chi}(a) m_\rho(\chi)}{s-\rho} + c_0 + c_1(s-\rho) + c_2(s-\rho)^2\ldots,\]
where $\sum_{\chi \imod{q}} \overline{\chi}(a) m_\rho(\chi) \neq 0$. In the same neighborhood, we therefore have
\[\begin{aligned}[b]
   \frac{d^2}{ds^2}\left( H_2(s)\right) &
   &= \frac{1}{\varphi(q)}\left(\frac{-1}{\rho}\frac{\sum_{\chi \imod{q}} \overline{\chi}(a) m_\rho(\chi)}{s-\rho} - \frac{1}{\rho-1} +\int_{1}^{\bar{p}_1} \! \frac{1}{t^\rho} \, dt + d_0 + d_1(s-\rho) + d_2(s-\rho)^2 + \ldots\right),
  \end{aligned}
\]
where $d_0, d_1, \ldots$ are coefficients arising from the Laurent expansion. Consequently, 
\[\begin{aligned}[b]\frac{d^2}{ds^2}\left( H_1(s) \right) 
&= \frac{1}{\varphi(q)}\left(\frac{\sum_{\chi \imod{q}} \overline{\chi}(a) m_\rho(\chi)}{\rho(s-\rho)^2} + d'_0 +d'_1(s-\rho)+ \ldots \right),
  \end{aligned}
\]
where $d'_0, d'_1, \ldots$ are constants. This implies that $\tfrac{d^2}{ds^2} H_1(s)$ has a pole of order 2 at $\rho$. 
However, we have claimed $\tfrac{d^2}{ds^2}\left(H_1(s) - H_2(s)\right)$ is holomorphic for $\Re(s) > 0$. This is a contradiction, and so $J(x;q,a)$ must \emph{not} be of constant sign on some interval $\left[ x', \infty \right)$. Hence we have established that $J(x;q,a)$ oscillates for arbitrary large $x$.
\end{proof}

We next establish a more precise $\Omega$-theorem for $J(x;q, a)$ under some assumptions on the location of the singularities of $\mathscr{L}(s;q,a)$.

\begin{theorem}\label{omegaJ2}
For $b>0$, suppose that $\mathscr{L}(s;q,a)$ has no singularities on the line segment $(1-b, 1)$. Assume that  there exists a singularity $\rho$ of $\mathscr{L}(s;q,a)$ not at $1$ for which $\Re(\rho) = \beta > 1-b$. Then
\[J(x;q,a) = \int_x^\infty \! R(t;q,a) g(t) \,dt = \Omega_{\pm}(x^{-b}).\]

\end{theorem}
\begin{proof}
We consider the integral
 \[G(s) = \int_{\bar{p}_1}^\infty \! \frac{J(x;q,a) - x^{-b}}{x^s} \, dx.\]
Then, for $\Re(s) >1$, \eqref{h1} and \eqref{H2} establish
\begin{equation}\label{GS}
  \begin{aligned}[b]
   G(s) &= \frac{1}{s-1}\left(H_1(s) - H_2(s) + E(s)\right) - \frac{1}{s-1+b} + \int_1^{\bar{p}_1} \! \frac{x^{-b}}{x^s} \, dx,
  \end{aligned}
 \end{equation}
where $H_1(s)$ and $H_2(s)$ are holomorphic in the region $W_{\delta}$. Furthermore, $1/(s-1+b)$, $E(s)$, and $\int_1^{\bar{p}_1} \! \frac{x^{-b}}{x^s} \, dx$ are holomorphic on
\[W_{\delta, b} = \{\,s \mbox{ ; } \Re(s) > 1-b \,\} \cap W_{\delta}.\] 
The right-hand side of \eqref{GS} therefore extends to a holomorphic function in the region $W_{\delta, b}$, shown in Figure \ref{fig:Wb}.

\begin{figure}\label{WB}
\centering
\begin{tikzpicture}
    \begin{scope}[thick,font=\scriptsize]
    \draw [->] (-0.5,0) -- (5.5,0) node [above left]  {$\Re(s)$};
    \draw [->] (0,-3) -- (0,3) node [below right] {$\Im(s)$};

    \draw (2,-3pt) -- (2,3pt)   node [above] {$1$};
   \draw (1.35,-3pt) -- (1.35,3pt)   node [below] {$1-b$};
    \draw (1.75,2.5) -- (1.75,2.5)   node [below] {$\rho$};
    \draw (0,-3pt) -- (0,3pt) node [above left] {$0$};
    \draw (-3pt,2) -- (3pt,2)   node [right] {$\delta$};
    \draw (-3pt,-2) -- (3pt,-2) node [right] {$-\delta$};

    \end{scope}
    \path [draw=none,fill=black] (1.75, 2.5) circle (1pt);
    \path [draw=none,fill=gray,semitransparent] (1.35,-2) rectangle (2,2);
    \path [draw=none,fill=gray,semitransparent] (2,-3) rectangle (5.5,3);
    \node [above right,gray] at (+2,+1) {$W_{\delta, b}$};
\end{tikzpicture}
\caption{The region $W_{\delta, b}$ and the singularity $\rho$.} \label{fig:Wb}
\end{figure}
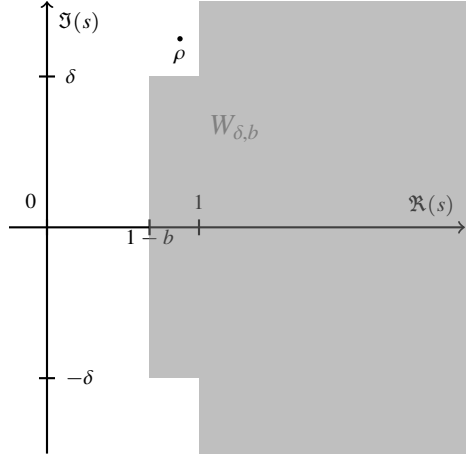

By  Theorem \ref{landauoscillation}, if we assume $J(x;q,a) - x^{-b}$ maintains a constant sign on intervals of the form $[x', \infty)$, then the abscissa of convergence $\sigma_0$ of $G(s)$ must satisfy $\sigma_0 \leq 1-b < \beta$. This is impossible since, as in the proof of Theorem \ref{omegaJ},  the second derivative of $H_1(s) - H_2(s)$ will have a pole of order 2 at $\rho$, contradicting the holomorphy of $G(s)$ in the half-plane $\Re(s) > 1-b$. We have a contradiction and therefore $J(x;q,a) - x^{-b} > 0$ on some sequence tending to infinity. Hence, 
\[J(x;q,a) = \Omega_{+}(x^{-b}).\]

Considering  $J(x;q,a) + x^{-b}$ and repeating the above proof establishes that $J(x;q,a) + x^{-b} < 0$ on another infinite sequence, i.e., \[J(x;q,a) = \Omega_{-}(x^{-b}).\]
\end{proof}

\section[Omega Theorems for $\log f(x;q,a)$]{Proofs of Theorems  \ref{BigUncond} and  \ref{big5}}
\label{sec:Omega Theoremsf}

\begin{proof}[Proof of Theorem \ref{BigUncond}]

By Proposition \ref{transferup}, we have, for $m= \text{Ind}_q(a)$,
\[\log{f(x;q,a)} =  J(x;q,a) - \frac{\mathcal{R}_{q,a}}{\varphi(q)} \frac{m}{(m-1)x^{\frac{m-1}{m}}\log{x}} + \bigO\left(\frac{1}{x^{\frac{m-1}{m}}\log^2{x}}\right).\]
We observe that  
$$0< \frac{\mathcal{R}_{q,a}}{\varphi(q)} \frac{m}{(m-1)x^{\frac{m-1}{m}}\log{x}}=\bigO\left(\frac{1}{x^{\frac{m-1}{m}}\log{x}}\right).$$
Therefore, for large enough $x$, 
\begin{equation*}\label{upperbounding}
 \log{f(x;q,a)} < J(x;q,a).
\end{equation*}

From this inequality the result follows, since  $J(x;q,a) < 0$ for arbitrarily large $x$ by Theorem \ref{omegaJ}.
\end{proof}

\begin{proof}[Proof of Theorem \ref{big5} (a)]
From Proposition \ref{transferup} we have
\[\log{f(x;q,a)} = J(x;q,a) + \bigO\left(\frac{1}{x^{\frac{m-1}{m}}\log{x}}\right).\]

Now since $\Theta\leq 1$ and $1-\Theta<b$, then by the assertion of Theorem \ref{omegaJ2}, there exists a sequence $\{x_i\}$ along which (i.e., for $x=x_i$)
\[J(x;q,a) > \frac{1}{x^b}.\]
Along that same sequence,
\begin{equation}\label{uplog}
\log{f(x;q,a)} > \frac{1}{x^b} + \bigO\left(\frac{1}{x^{\frac{m-1}{m}}\log{x}}\right) = \frac{1}{x^b}\left(1 + \bigO\left(\frac{1}{x^r \log{x}}\right)\right),
\end{equation}
where $r = \tfrac{m-1}{m} - b > 0$ since $b < 1/2$ and $(m-1)/m \geq 1/2$. Hence, the bracketed expression may be bounded by a positive constant, establishing
\[\log{f(x;q,a)} = \Omega_{+}(x^{-b}).\]

The proof for $\log{f(x;q,a)} = \Omega_{-}(x^{-b})$ follows similarly by employing $J(x;q,a) = \Omega_{-}(x^{-b})$. 
\end{proof}

\begin{proof}[Proof of Theorem \ref{big5} (b)]

As in the proof of Theorem \ref{big5} (a), we know that 
\[\log{f(x;q,a)} = J(x;q,a) + \bigO\left(\frac{1}{x^{\frac{m-1}{m}}\log{x}}\right),\]
where $m$ is at least 3.  Also the conditions of Theorem \ref{omegaJ2} for $b>1/2$ holds. Therefore, there exists a sequence $\{x_i\}$ along which
\[\log{f(x;q,a)} > \frac{1}{x^b}\left(1+ \bigO\left(\frac{1}{x^r \log{x}}\right)\right),\]
where $r = \tfrac{m-1}{m} - b > 0$ since $b < 2/3$ and $(m-1)/m \geq 2/3$. This establishes
\[\log{f(x;q,a)} = \Omega_{+}(x^{-b})\]
and the proof for $\log{f(x;q,a)} = \Omega_{-}(x^{-b})$ follows in a similar manner. 
\end{proof}

\section{An Explicit Formula for $\log{f(x;q, a)}$}\label{GRH True}

We start with a version of the explicit formula for $\int_{1}^{x}\psi(t;q,a) dt$. 

\begin{theorem}[{\cite[Lemma 3.1]{mccurley1984}}]\label{explicitformula}
For a Dirichlet character $\chi$ modulo $q$, write
\[\mathcal{Z}(\chi) = \{ \rho \in \C \text{ ; } L(\rho,\chi) = 0 \text{, } \Re(\rho) \geq 0 \text{ and } \rho \neq 0 \}.\]
Let $\alpha$ be $1$ if $\chi$ is odd and $0$ otherwise, $b(\chi)$, $c(\chi)$ be the constant terms in the Laurent expansion of $\frac{L'}{L}(s,\chi)$ about 0 and -1, respectively, and $m_0(\chi)$ be the multiplicity of the zero of $L(s,\chi)$ at 0.
Then, for $x > 1$, we have

 \begin{align*}
\int_1^x \! \psi(t;q,a) \, dt =& \frac{x^2}{2\varphi(q)}  - \frac{1}{\varphi(q)}\sum_{\chi \imod{q}}\overline{\chi}(a)\sum_{\rho \in \mathcal{Z}(\chi)} \frac{x^{\rho +1}}{\rho(\rho+1)} +\hat{R}(x;q,a),
\end{align*}
where
\begin{align}
\hat{R}(x;q,a)=&- \frac{1}{\varphi(q)} \sum_{\chi \imod{q}} \overline{\chi}(a) \sum_{n=1}^\infty \frac{x^{-2n+1-\alpha}}{2n(2n-1+2\alpha)}\notag \\ 
&+ x \left( \frac{1}{\varphi(q)} \sum_{\chi \imod{q}} \overline{\chi}(a)(m_0(\chi) - b(\chi))\right) \notag\\ 
&- x\log{x} \left(\frac{1}{\varphi(q)} \sum_{\chi \imod{q}} \overline{\chi}(a)m_0(\chi)\right) + \log{x}\left(\frac{1}{\varphi(q)} \sum_{\chi \textup{ odd}}\overline{\chi}(a)\right) \notag\\ \label{ghat}
&+ \left( \frac{1}{\varphi(q)} \sum_{\chi \textup{ even}}\overline{\chi}(a)\frac{L'}{L}(-1, \chi) + \frac{1}{\varphi(q)} \sum_{\chi \textup{ odd}}\overline{\chi}(a)(c(\chi) +1) \right).\notag
 \end{align}

\end{theorem}

Under the assumption of SH${}_{q, a}$ the explicit formula given in Theorem \ref{explicitformula} allows us to establish a new expression for $J(x;q,a)$. 

\begin{lemma}\label{ghatprime}
Suppose that SH${}_{q, a}$ is true. Then, for $x > 1$, we have
 \[J(x;q,a) = -\frac{1}{\varphi(q)}\sum_{\chi \imod{q}}\overline{\chi}(a) \sum_{\rho \in \mathcal{Z}(\chi)} \frac{F_{\rho}(x)}{\rho} + \hat{J}(x;q,a),\]
where 
\[F_{\rho}(x) = \int_x^\infty \! t^\rho g(t) \, dt,\]
\[\hat{J}(x;q,a) = \int_x^\infty \! \hat{R}'(t;q,a)g(t) \, dt\]
with 
 \[g(t) = \frac{\log{t}+1}{t^2 \log^2{t}},\]
and
 \[
  \begin{aligned}[b]
   \hat{R}'(t;q,a) = &-\frac{1}{\varphi(q)}\sum_{\chi \imod{q}}\overline{\chi}(a)\sum_{n=1}^\infty \left(\frac{-t^{-(2n+\alpha)}}{2n+\alpha}\right) + \frac{1}{\varphi(q)} \sum_{\chi \imod{q}}\overline{\chi}(a)(m_0(\chi) - b(\chi))\notag \\
   &-\frac{1}{\varphi(q)}\sum_{\chi \imod{q}} \overline{\chi}(a)m_0(\chi)(\log{t} +1) + \frac{1}{\varphi(q)} \sum_{\chi \textup{ odd}} \overline{\chi}(a)t^{-1}.
  \end{aligned}
 \]

\end{lemma}

\begin{proof}

Since we are assuming SH${}_{q, a}$, recall that for any zero $\rho$ in $\bigcup_\chi \mathcal{Z}(\chi)$ whose real part is neither $0$ nor $\tfrac{1}{2}$, we have that
\begin{equation}\label{altRH}
\sum_{\chi \imod{q}} \overline{\chi}(a) m_\rho(\chi) = 0.
\end{equation}
Thus, this assumption allows us to prove that
\[\sum_{\chi \imod{q}} \overline{\chi}(a) \sum_{\rho \in \mathcal{Z}(\chi)} g'(t) \frac{t^{\rho+1}}{\rho(\rho+1)}\]
is integrable on $(x,\infty)$ for any $x>1$ and it can be integrated term by term. In anticipation of an application of the Dominated Convergence Theorem, for $n \geq 1$, consider the sequence of functions
\[f_n(t) = \sum_{\chi \imod{q}} \overline{\chi}(a)\sum_{\substack{\rho \in \mathcal{Z}(\chi) \\ \lvert \Im(\rho) \rvert \leq n}} g'(t) \frac{t^{\rho + 1}}{\rho(\rho+1)}.\]
We see that, for all $t \in (x,\infty)$, 
\[\lim_{n \to \infty} f_n(t) = \sum_{\chi \imod{q}} \overline{\chi}(a)\sum_{\rho \in \mathcal{Z}(\chi)} g'(t) \frac{t^{\rho +1}}{\rho(\rho+1)}.\]
Moreover,
\[\begin{aligned}[b]
  &\lvert f_n(t) \rvert = \left\lvert g'(t) \sum_{\chi \imod{q}} \overline{\chi}(a)\sum_{\substack{\rho \in \mathcal{Z}(\chi) \\ \lvert \Im(\rho) \rvert \leq n}} \frac{t^{\rho + 1}}{\rho(\rho+1)} \right\rvert \\ 
  &\leq \lvert g'(t) \rvert \cdot \left\lvert \sum_{\chi \imod{q}} \overline{\chi}(a)\left(\sum_{\substack{\rho \in \mathcal{Z}(\chi) \\ \Re(\rho) = 1/2 \\ \lvert \Im(\rho) \rvert \leq n}} \frac{t^{\rho +1}}{\rho(\rho+1)} + \sum_{\substack{\rho \in \mathcal{Z}(\chi) \\ \Re(\rho) = 0 \\ \lvert \Im(\rho) \rvert \leq n}} \frac{t^{\rho +1}}{\rho(\rho+1)}\right) + \sum_{\chi \imod{q}} \overline{\chi}(a) \sum_{\substack{\rho \in \mathcal{Z}(\chi) \\ \Re(\rho) \neq 0,1/2 \\ \lvert \Im(\rho) \rvert \leq n}} \frac{t^{\rho +1}}{\rho(\rho+1)}\right\rvert.
  \end{aligned}
\]
Since zeroes are counted with multiplicity, the final sum becomes 
\[
 \sum_{\substack{\rho \in \cup_\chi \mathcal{Z}(\chi) \\ \Re(\rho) \neq 0, 1/2 \\ \lvert \Im(\rho) \leq n}} \frac{\sum_{\chi \imod{q}} \overline{\chi}(a) m_\rho(\chi)}{\rho (\rho+1)} t^{\rho+1} = 0
\]
by way of \eqref{altRH}. We now have
\[
  \lvert f_n(t) \rvert \leq \lvert g'(t) \rvert \left(t^\frac{3}{2} \sum_{\chi \imod{q}} \sum_{\substack{\rho \in \mathcal{Z}(\chi) \\ \Re(\rho) = 1/2 \\ \lvert \Im(\rho) \rvert \leq n}} \frac{1}{\lvert \rho(\rho+1) \rvert}  + t \sum_{\chi \imod{q}} \sum_{\substack{\rho \in \mathcal{Z}(\chi) \\ \Re(\rho) = 0 \\ \lvert \Im(\rho) \rvert \leq n}} \frac{1}{\lvert \rho(\rho+1) \rvert} \right).
  \]
Removing the restriction on $\Im(\rho)$ yields
  \[
  \lvert f_n(t) \rvert  \leq \lvert g'(t) \rvert \left(t^\frac{3}{2} \sum_{\chi \imod{q}} \sum_{\substack{\rho \in \mathcal{Z}(\chi) \\ \Re(\rho) = 1/2 }} \frac{1}{\lvert \rho(\rho+1) \rvert} + t \sum_{\chi \imod{q}} \sum_{\substack{\rho \in \mathcal{Z}(\chi) \\ \Re(\rho) = 0}} \frac{1}{\lvert \rho(\rho+1) \rvert} \right).
\]
We note that this upper bound for $\lvert f_n(t) \rvert$ is integrable on $(x,\infty)$ for $x>1$, since 
\[g'(t) = -\frac{1}{t^3}\left(\frac{2}{\log{t}}+\frac{3}{\log^2{t}}+\frac{2}{\log^3{t}}\right).\]
Hence, the dominated convergence theorem allows us to have
\begin{equation}\label{onetwo}
\int_x^\infty \! \frac{1}{\varphi(q)} \sum_{\chi \imod{q}} \overline{\chi}(a)\sum_{\rho \in \mathcal{Z}(\chi)} g'(t) \frac{t^{\rho+1}}{\rho(\rho+1)} \, dt = \frac{1}{\varphi(q)} \sum_{\chi \imod{q}} \overline{\chi}(a) \Big( \sum_{\rho \in \mathcal{Z}(\chi)} \int_x^\infty \! g'(t) \frac{t^{\rho+1}}{\rho(\rho+1)} \, dt \Big). 
\end{equation}

Now employing integration by part on the left-hand side of \eqref{onetwo}  together with Theorem \ref{explicitformula} yield
\begin{eqnarray}
\int_x^\infty \! \frac{1}{\varphi(q)} \sum_{\chi \imod{q}} \overline{\chi}(a)\sum_{\rho \in \mathcal{Z}(\chi)} g'(t) \frac{t^{\rho+1}}{\rho(\rho+1)} \, dt&=& 0- 
 \frac{g(x)}{\varphi(q)} \sum_{\chi \imod{q}} \overline{\chi}(a) \sum_{\rho \in \mathcal{Z}(\chi)} \frac{x^{\rho+1}}{\rho(\rho+1)} \notag\\ \label{oneone}
 &&+\int_{x}^{\infty} g(t)\left(\psi(t;q,a) - \frac{t}{\varphi(q)} \right) \, dt -\int_x^\infty \! g(t)\hat{R}'(t;q,a) \, dt.
\end{eqnarray}

Similarly integration by parts on the right-hand side of \eqref{onetwo} yields
\begin{eqnarray}
 \frac{1}{\varphi(q)} \sum_{\chi \imod{q}} \overline{\chi}(a) \Big( \sum_{\rho \in \mathcal{Z}(\chi)} \int_x^\infty \! g'(t) \frac{t^{\rho+1}}{\rho(\rho+1)} \, dt \Big) &=&0- 
 \frac{g(x)}{\varphi(q)} \sum_{\chi \imod{q}} \overline{\chi}(a) \sum_{\rho \in \mathcal{Z}(\chi)} \frac{x^{\rho+1}}{\rho(\rho+1)} \notag\\ \label{twotwo}
 &&-\frac{1}{\varphi(q)}\sum_{\chi \imod{q}} \overline{\chi}(a) \sum_{\rho \in \mathcal{Z}(\chi)} \frac{F_{\rho}(x)}{\rho},
\end{eqnarray}
where $F_{\rho}(x) = \int_x^\infty \! g(t)t^\rho \, dt$.

The result follows by \eqref{onetwo}, \eqref{oneone}, and \eqref{twotwo}.
\end{proof}
\begin{corollary}\label{corr}
Under the assumption of SH${}_{q, a}$
\begin{equation}\label{jident2}
J(x;q,a) = \frac{1}{\varphi(q) \sqrt{x} \log{x}} \sum_{\chi \imod{q}} \overline{\chi}(a) \sum_{\rho \in \mathcal{Z}(\chi')} \frac{x^{i\Im(\rho)}}{\rho(\rho-1)} + \hat{J}(x;q,a) + \bigO\left(\frac{1}{\sqrt{x}\log^2{x}}\right).
\end{equation}
\end{corollary}
\begin{proof}

Applying Lemma \ref{nicolaslem} to $F_\rho(x)$ in Lemma \ref{ghatprime} yields
\[J(x;q,a) = \frac{1}{\varphi(q)}\left(\sum_{\chi \imod{q}} \overline{\chi}(a) \sum_{\rho \in \mathcal{Z}(\chi)}\left(\frac{x^{\rho-1}}{\rho(\rho-1)\log{x}} - \frac{r_\rho(x)}{\rho}\right)\right) + \hat{J}(x;q,a).\]
If we apply the estimates of \eqref{rs} to ${r_\rho(x)}$ at strictly imaginary zeroes arising from an imprimitive character, we see that these zeroes only contribute terms of order $\bigO(\tfrac{1}{x\log{x}})$ to the sum involving such zeroes. Meanwhile, since we have assumed SH${}_{q, a}$, any other contribution must be from zeroes of the form $\rho = 1/2 + it$, and therefore the estimate  \eqref{rs} for $r_\rho(x)$ yields the result.
\end{proof}

The next lemma provides an estimation for  $\hat{J}(x;q,a)$.
\begin{lemma}\label{lemmm}
We have $\hat{J}(x;q,a)=O(1/x).$
\end{lemma}
\begin{proof}
In order to bound $\hat{J}(x;q,a)$, we aim to bound $\hat{R}'(t;q,a)$ in absolute value. If we write 
\[\nu(k;q,a) = 
\begin{cases}
\displaystyle{\sum_{\chi \textup{ odd}}}\overline{\chi}(a) & \textnormal{if } k \textup{ is odd}, 
\\
\displaystyle{\sum_{\chi \textup{ even}}}\overline{\chi}(a) &\textnormal{if } k \textrm { is even},
\end{cases}
\]
then we may observe that 
\begin{equation}\label{sumtolog}
 \sum_{\chi \textup{ odd}} \overline{\chi}(a) \sum_{n=1}^\infty \frac{t^{-2n-1}}{2n+1} + \sum_{\chi \textup{ even}} \overline{\chi}(a) \sum_{n=1}^\infty \frac{t^{-2n}}{2n} + \sum_{\chi \textup{ odd}}\overline{\chi}(a) t^{-1} = \sum_{k=1}^\infty \frac{\nu(k;q,a) t^{-k}}{k},
\end{equation}
where the terms on the left all appear in the expression for $\hat{R}'(t;q,a)$ in Lemma \ref{ghatprime}. 
Since there are $\tfrac{\varphi(q)}{2}$ even and $\tfrac{\varphi(q)}{2}$ odd characters respectively, $\lvert \nu(x;q,a) \rvert \leq \tfrac{\varphi(q)}{2}$ and therefore, for $t >1$,
\[\left| \sum_{k=1}^\infty \frac{\nu(k;q,a) t^{-k}}{k} \right|  \leq \frac{\varphi(q)}{2} \sum_{k=1}^\infty \frac{t^{-k}}{k} = -\frac{\varphi(q)}{2}\log{\left(1-\frac{1}{t}\right)}.\]
Observe that $-\log(1-\frac{1}{t})$ is always positive and decreasing on $(1,\infty)$, and therefore for $t > e^4$,
\begin{equation}\label{hatboundsum}
 \begin{aligned}[b]
  &\frac{1}{\varphi(q)}\left[\sum_{\chi \textup{ odd}} \overline{\chi}(a) \sum_{n=1}^\infty \frac{t^{-2n-1}}{2n+1} + \sum_{\chi \textup{ even}} \overline{\chi}(a) \sum_{n=1}^\infty \frac{t^{-2n}}{2n} + \sum_{\chi \textup{ odd}}\overline{\chi}(a) t^{-1} \right] \\
  &\leq -\frac{1}{2}\log{\left(1-\frac{1}{t}\right)} \leq  -\frac{1}{2}\log{\left(1-\frac{1}{e^4}\right)} \leq 0.01.
 \end{aligned}
\end{equation}

Returning to expression for $\hat{R}'(t;q,a)$ in Lemma  \ref{ghatprime}, we can take the absolute value and apply \eqref{hatboundsum} to determine that for $t > e^4$,
\begin{equation}\label{hatboundtotal}
 \begin{aligned}[b]
  \lvert \hat{R}'(t;q,a) \rvert &\leq \frac{1}{\varphi(q)}\left |\sum_{\chi \textup{ odd}} \overline{\chi}(a) \sum_{n=1}^\infty \frac{t^{-2n-1}}{2n+1} + \sum_{\chi \textup{ even}} \overline{\chi}(a) \sum_{n=1}^\infty \frac{t^{-2n}}{2n} + \sum_{\chi \textup{ odd}}\overline{\chi}(a) t^{-1} \right| \\ 
  &+ \frac{1}{\varphi(q)}\left|\sum_{\chi \imod{q}} \overline{\chi}(a) b(\chi) \right| + \frac{1}{\varphi(q)} \left| \sum_{\chi \imod{q}} \overline{\chi}(a) m_0(\chi) \log{t}\right| \\
  &\leq 0.01 + \frac{\mathcal{B}_q}{\varphi(q)} + \frac{\mathcal{M}_q \log{t}}{\varphi(q)},
 \end{aligned}
\end{equation}
where $ \mathcal{B}_{q}{:=}\sum_{\chi \imod{q}} \lvert b(\chi) \rvert $ and $\mathcal{M}_{q} {:=} \sum_{\chi \imod{q}} m_0(\chi)$. 
With this bound in place, we may now turn our attention to estimating $\hat{J}(x;q,a)$. For $x>e^4$, we have, by \eqref{hatboundtotal},
\begin{equation}\label{Jhateasy}
\begin{aligned}[b]
 \left| \hat{J}(x;q,a) \right| &= \left| \int_x^\infty \! \hat{R}'(t;q,a)\, d\left(\frac{-1}{t\log{t}}\right) \right| 
 &\leq \int_x^\infty \! \left( 0.01 + \frac{\mathcal{B}_q}{\varphi(q)} \right)\, d\left(\frac{-1}{t\log{t}}\right) + \int_x^\infty \! \frac{\mathcal{M}_q \log{t}}{\varphi(q)}\, d\left(\frac{-1}{t\log{t}}\right).
\end{aligned}
\end{equation}
Evaluating integrals in \eqref{Jhateasy} yields
\begin{equation}\label{Jhattotal}
 \left| \hat{J}(x;q,a) \right| \leq \frac{0.01 \varphi(q) + \mathcal{B}_q + \mathcal{M}_q}{\varphi(q)x\log{x}} + \frac{\mathcal{M}_q}{\varphi(q)x}.
\end{equation}

This implies that $\hat{J}(x;q,a) = \bigO(1/x)$.
\end{proof}
\begin{proof}[Proof of Theorem \ref{oscillation}]
The result follows by applying Corollary \ref{corr} and Lemma \ref{lemmm} to Proposition \ref{transferup}.
\end{proof}

\section{Computation of $\mathcal{F}_q$}\label{comp}

Consider
\begin{equation}\label{truefq}
\mathcal{F}_q = \sum_{\chi \imod{q}} \mathcal{F}(\chi),
\end{equation}
where, 
\[\mathcal{F}(\chi) = \sum_{\rho \in \mathcal{Z}(\chi')} \frac{1}{\rho(1-\rho)},\]
recalling that $\mathcal{Z}(\chi) = \{ \rho \in \C \text{ ; } L(\rho,\chi) = 0 \text{, } \Re(\rho) \geq 0 \text{ and } \rho \neq 0 \}$. 

 Let $q =1$. Corollary 10.14 of \cite{montgomery2007} establishes
 \begin{equation}\label{f1}
\mathcal{F}_1 = \sum_\rho \frac{1}{\rho(1-\rho)} = 2 + \gamma -\log{\pi} - 2\log{2} \approx 0.04619,
 \end{equation}
 where the sum is over the non-trivial zeroes of $\zeta(s)$. 

For larger $q$, we keep \eqref{f1} in mind, since the principal character modulo $q$ will always be induced by the trivial character, and therefore 
\[\sum_{\rho \in \mathcal{Z}(\chi_{0}')} \frac{1}{\rho(1-\rho)} = \mathcal{F}_1\]
for any $q$. If $\chi$ is not principal, its contribution to $\mathcal{F}_{q}$ is determined by the results of Corollary 10.18 of \cite{montgomery2007} which determine that
\begin{equation}\label{Fcalc}
 \sum_{\rho \in \mathcal{Z}(\chi')} \frac{1}{\rho(1-\rho)} =  \log{\frac{q}{\pi}} + 2\Re(\frac{L'}{L}(1,\overline{\chi'})) - \gamma - (1-\alpha)2\log{2},
\end{equation}
where $\alpha=1$ if $\chi$ is odd and $\alpha=0$ if $\chi$ is even.

Suppose $q = p$ is an odd prime, so that all the nonprincipal characters modulo $p$ are primitive. Then, summing \eqref{Fcalc} over all characters yields
\begin{equation}\label{primecalc}
 \mathcal{F}_{p} = \mathcal{F}_{1} + 2\gamma_p - {p}{\gamma} + (2-p)\log\left(\frac{2\pi}{p}\right) + \log{2}, 
\end{equation}
where $$\gamma_p=\gamma+\sum_{\chi \neq \chi_0}  \frac{L'}{L}(1,{\chi})$$ is the Euler-Kronecker constant associated with the cyclotomic field $\Q(e^{2\pi i / p})$. Computations of the value of $\gamma_p$ are provided in \cite{ford2014}. Using these and \eqref{primecalc} we determine the value of $\mathcal{F}_p$ for odd primes up to $p = 149$. Several of these values are listed in Table \ref{Fprime}.

\begin{table}[ht]
\begin{tabular}{rrr}
\rowcolor[HTML]{EFEFEF} 
\multicolumn{1}{c|}{\cellcolor[HTML]{EFEFEF}$p$} & \multicolumn{1}{c|}{\cellcolor[HTML]{EFEFEF}{\color[HTML]{000000} $\gamma_p$}} & \multicolumn{1}{c}{\cellcolor[HTML]{EFEFEF}{\color[HTML]{000000} $\mathcal{F}_p$}} \\ \hline
\multicolumn{1}{r|}{3}                           & \multicolumn{1}{r|}{0.94550}                                                   & 0.15942                                                                            \\
\multicolumn{1}{r|}{5}                           & \multicolumn{1}{r|}{1.72062}                                                   & 0.60919                                                                            \\
\multicolumn{1}{r|}{7}                           & \multicolumn{1}{r|}{2.08759}                                                   & 1.41418                                                                            \\
\multicolumn{1}{r|}{11}                          & \multicolumn{1}{r|}{2.41542}                                                   & 4.26098                                                                            \\
\multicolumn{1}{r|}{13}                          & \multicolumn{1}{r|}{2.61076}                                                   & 6.45484                                                                            \\
\multicolumn{1}{r|}{17}                          & \multicolumn{1}{r|}{3.58198}                                                   & 13.02067                                                                           \\
\multicolumn{1}{c}{$\vdots$}                     & \multicolumn{1}{c}{$\vdots$}                                                   & \multicolumn{1}{c}{$\vdots$}                                                       \\
\multicolumn{1}{r|}{139}                         & \multicolumn{1}{r|}{5.88917}                                                   & 356.51847                                                                          \\
\multicolumn{1}{r|}{149}                         & \multicolumn{1}{r|}{5.98342}                                                   & 392.11323                                                                         
\end{tabular}
\caption{Some values of $\mathcal{F}_p$.}
\label{Fprime}
\end{table}

Now, suppose $q = 2p$, where $p$ is either an odd prime or 1. Then all of the characters mod $q$ are induced by the characters mod $p$, and so 
\[\mathcal{F}_{2p} = \mathcal{F}_p.\]

For $q = 4, 8,$ and $12$, the matter of computing $\mathcal{F}(\chi)$ via \eqref{Fcalc} has been left to the Python package \textsc{mpmath} \cite{mpmath}, in particular for the computation of the logarithmic derivative $\frac{L'}{L}(1,\overline{\chi'})$. We include values related to \eqref{Fcalc} towards the computation of $\mathcal{F}_4$, $\mathcal{F}_8$, and $\mathcal{F}_{12}$ in Tables \ref{fq4}, \ref{fq8}, and \ref{fq12}, respectively .  The numbering of the characters follows \cite{lmfdb}. 

\begin{table}[H]
\centering
\begin{tabular}{|c|lll|}
\hline
\rowcolor[HTML]{EFEFEF} 
$\chi$ & \multicolumn{1}{c}{\cellcolor[HTML]{EFEFEF}$\alpha$} & \multicolumn{1}{c}{\cellcolor[HTML]{EFEFEF}$\Re(\frac{L'}{L}(1,\overline{\chi}))$} & \multicolumn{1}{c|}{\cellcolor[HTML]{EFEFEF}$\mathcal{F}(\chi)$} \\ \hline
$\chi_1(1,\cdot)$                       & \multicolumn{1}{l|}{0} & \multicolumn{1}{c|}{}          & 0.0461914 \\ \hline
$\chi_4(3,\cdot)$                       & \multicolumn{1}{l|}{1} & \multicolumn{1}{l|}{0.2456096} & 0.1555680 \\ \hline
\cellcolor[HTML]{EFEFEF}$\mathcal{F}_4$ & \multicolumn{3}{c|}{0.2017594} \\ \hline
\end{tabular}
\caption{Values relevant to the computation of $\mathcal{F}_4$.}
\label{fq4}
\end{table}

\begin{table}[H]
\centering
\begin{tabular}{|c|lll|}
\hline
\rowcolor[HTML]{EFEFEF} 
$\chi$ & \multicolumn{1}{c}{\cellcolor[HTML]{EFEFEF}$\alpha$} & \multicolumn{1}{c}{\cellcolor[HTML]{EFEFEF}$\Re(\frac{L'}{L}(1,\overline{\chi}))$} & \multicolumn{1}{c|}{\cellcolor[HTML]{EFEFEF}$\mathcal{F}(\chi)$} \\ \hline
$\chi_1(1,\cdot)$                       & \multicolumn{1}{l|}{0} & \multicolumn{1}{c|}{}          & 0.0461914 \\ \hline
$\chi_8(3,\cdot)$                       & \multicolumn{1}{l|}{1} & \multicolumn{1}{l|}{-0.0207114} & 0.3160732 \\ \hline
\multicolumn{1}{|l|}{$\chi_8(5,\cdot)$} & \multicolumn{1}{l|}{0} & \multicolumn{1}{l|}{0.6321150} & 0.2354316 \\ \hline
\multicolumn{1}{|l|}{$\chi_4(3,\cdot)$} & \multicolumn{1}{l|}{1} & \multicolumn{1}{l|}{0.2456096} & 0.1555680 \\ \hline
\cellcolor[HTML]{EFEFEF}$\mathcal{F}_8$ & \multicolumn{3}{c|}{0.7532641} \\ \hline
\end{tabular}
\caption{Values relevant to the computation of $\mathcal{F}_8$.}
\label{fq8}
\end{table}

\begin{table}[H]
\centering
\begin{tabular}{|c|lll|}
\hline
\rowcolor[HTML]{EFEFEF} 
$\chi$ & \multicolumn{1}{c}{\cellcolor[HTML]{EFEFEF}$\alpha$} & \multicolumn{1}{c}{\cellcolor[HTML]{EFEFEF}$\Re(\frac{L'}{L}(1,\overline{\chi}))$} & \multicolumn{1}{c|}{\cellcolor[HTML]{EFEFEF}$\mathcal{F}(\chi)$} \\ \hline
\multicolumn{1}{|c|}{$\chi_1(1,\cdot)$}                      & \multicolumn{1}{l|}{0} & \multicolumn{1}{c|}{}          & 0.0461914 \\ \hline
\multicolumn{1}{|c|}{$\chi_3(2,\cdot)$}                       & \multicolumn{1}{l|}{1} & \multicolumn{1}{l|}{0.3682816} & 0.1132300 \\ \hline
\multicolumn{1}{|c|}{$\chi_4(3,\cdot)$} & \multicolumn{1}{l|}{1} & \multicolumn{1}{l|}{0.2456096} & 0.1555680 \\ \hline
\multicolumn{1}{|c|}{$\chi_{12}(11,\cdot)$} & \multicolumn{1}{l|}{0} & \multicolumn{1}{l|}{0.4767499} & 0.3301666 \\ \hline
\cellcolor[HTML]{EFEFEF}$\mathcal{F}_{12}$ & \multicolumn{3}{c|}{0.6451560} \\ \hline
\end{tabular}
\caption{Values relevant to the computation of $\mathcal{F}_{12}$.}
\label{fq12}
\end{table}

Finally, for $q = 9$ we implemented a naive version of the methods suggested in \cite[Section 3.2]{ford2014} to compute $L(1, \chi)$ and $L^\prime(1, \chi)$, for primitive characters $\chi$ modulo $3$ and $9$. Then we applied \eqref{fq} to compute $\mathcal{F}_9$. 

\section{Proof of Theorem \ref{criterion}}\label{eight}

\begin{proof}
Let $q \leq 10$ or $q = 12, 14$ and $a = 1$. Recall that the satisfaction of the inequality 
\[\frac{\bar{N}_k}{\varphi(\bar{N}_k)(\log(\varphi(q)\log{\bar{N}_k}))^{\frac{1}{\varphi(q)}}} > \frac{1}{C(q,1)}\]
for all positive integers $k$ is equivalent to $\log{f(x;q,1)} < 0$ for all $x > 1$. For $q = 1$, Theorem \ref{criterion} is exactly Theorem 2 of \cite{nicolas1983}.  

For $q > 1$, it is a consequence of part (a) of Theorem \ref{big5} that if $\log{f(x;q,1)} < 0$ for all $x$, then SH${}_{q, a}$ is true for the given $q$ and $a$. Since $a = 1$, SH${}_{q, a}$ implies (and in fact is equivalent to) $\textrm{GRH}_q$. Hence, to establish Theorem \ref{criterion}, we only need to show that if $\textrm{GRH}_q$ is true, then $\log{f(x;q,a)} < 0$ for all $x$. In the case of $q = 2$, observe that $f(x;1,1) > f(x;2,1)$ since $C(2,1) = 2{C(1,1)}$ and therefore the work of Nicolas shows that $1 > f(x;1,1) > f(x;2,1)$, and hence Theorem \ref{criterion} holds in the case $q = 2$.  

For the remaining moduli, \eqref{geou} and \eqref{inception} imply that for $x > 1$
\begin{equation}
\label{K1}
 \log{f(x;q,1)} \leq K(x;q,1) + \frac{1}{2(x-1)},
\end{equation}
since $S(x;q,a)\left(\tfrac{x\log{x} - c\log{c}}{(x\log{x})(c\log{c})}\right)$ is always negative. Moreover, if $x_q$ is defined as the smallest $x$ for which 
\begin{equation}
\label{xq}
\frac{\theta(x^\frac{1}{2};q,b)}{x^\frac{1}{2}} > \frac{0.6}{\phi(q)}
\end{equation}
for all $b$ in $\left(\Z/q\Z\right)^\times$, then \eqref{tetasi} implies that for $x \geq x_q$,
\begin{equation}
 \label{K2}
 \theta(x;q,1) \leq \psi(x;q,1) - \frac{0.6\mathcal{R}_{q,1}}{\phi(q)}x^{\frac{1}{2}}.
\end{equation}
It follows from \eqref{K2} that 
\[K(x;q,1) \leq J(x;q,1) - \frac{0.6\mathcal{R}_{q,1}}{\phi(q)}F_{\frac{1}{2}}(x),\]
which, with \eqref{K1}, establishes
\begin{equation}
\label{logf}
\log{f(x;q,1)} \leq  J(x;q,1) - \frac{ 0.6\mathcal{R}_{q,1}}{\phi(q)}F_{\frac{1}{2}}(x) + \frac{1}{2(x-1)},
\end{equation}
for $x \geq x_q$. 

We aim to show that the right-hand side of inequality \eqref{logf} is negative. We start by establishing an explicit upper bound for $J(x;q,1)$. Recall Lemma \ref{ghatprime} and observe that
\begin{equation*}\label{splitmainJ}
 -\frac{1}{\phi(q)}\sum_{\chi \imod{q}}\overline{\chi}(a)\sum_{\rho \in \mathcal{Z}(\chi)} \frac{F_{\rho}(x)}{\rho} 
 = \frac{1}{\phi(q)} \left( \sum_{\chi \imod{q}} \overline{\chi}(a) \sum_{\rho \in \mathcal{Z}(\chi')}\frac{-F_\rho(x)}{\rho} + \sum_{\chi \neq \chi'} \overline{\chi}(a) \sum_{\substack{\rho \in \mathcal{Z}(\chi) \\ \Re(\rho) = 0}} \frac{-F_\rho(x)}{\rho}\right),
\end{equation*}
where $\chi'$ is the character which induces $\chi$. Since we are assuming $\textrm{GRH}_q$ we have that $\Re(\rho)=1/2$ in the sum involving $\rho\in Z(\chi^\prime)$. Now employing \eqref{rs} to estimate  $r_{1/2}(x)$ and $r_0(x)$ yields
 \begin{equation}\label{ro}
\left| -\frac{1}{\phi(q)}\sum_{\chi \imod{q}}\overline{\chi}(a)\sum_{\rho \in \mathcal{Z}(\chi)} \frac{F_{\rho}(x)}{\rho}\right| \leq
 \frac{1}{\phi(q)}\left(\frac{\left(1+\frac{3}{\log{x}}\right)\mathcal{F}_q}{\sqrt{x}\log{x}} + \frac{\left(1+\frac{2}{\log{x}}\right)\mathcal{G}_q }{x\log{x}}\right),
 \end{equation}
for $x>e^2$, where $\mathcal{F}_q$ is defined in \eqref{fqdef} and \[\mathcal{G}_q = \sum_{\chi \imod{q}} \sum_{\substack{\rho \in \mathcal{Z}(\chi) \\ \Re(\rho) = 0}} \frac{1}{\lvert \rho(1-\rho)\rvert}.\]
We can compute $\mathcal{G}_q$ directly in Maple by recalling that all of the zeroes in the sum are in an arithmetic progression along the imaginary axis. 
Furthermore, in Lemma \ref{lemmm} if we are less zealous with our use of absolute values in arriving at \eqref{Jhattotal} we determine that if $x > e^4$ we have
\begin{equation}\label{hat}
\hat{J}(x;q,1) \leq \frac{ 0.01\phi(q) - \mathcal{B}_q - \mathcal{M}_q}{x\log{x}} - \frac{\mathcal{M}_q}{x}.
\end{equation}
Now by applying \eqref{ro} and \eqref{hat} in Lemma \ref{ghatprime} we conclude that under the assumption of  $\textrm{GRH}_q$  for $x>\max\{e^4, x_q\}$ we have
 \begin{equation}\label{thesisbound}
  J(x;q,1) \leq \frac{1}{\phi(q)}\left(\frac{\left(1+\frac{3}{\log{x}}\right)\mathcal{F}_q}{\sqrt{x}\log{x}} + \frac{\left(1+\frac{2}{\log{x}}\right)\mathcal{G}_q }{x\log{x}}+\frac{ 0.01\phi(q) - \mathcal{B}_q - \mathcal{M}_q}{x\log{x}} - \frac{\mathcal{M}_q}{x}\right).
 \end{equation}

Nicolas \cite[(2.4)]{nicolas2012} establishes that 
\[-F_{\frac{1}{2}}(x) \leq - \frac{2}{\sqrt{x}\log{x}} + \frac{2}{\sqrt{x}\log^2{x}}\]
for $x > 1$ and therefore, with \eqref{logf} and \eqref{thesisbound}, we arrive at an upper bound for $\log{f(x;q,1)}$ under $\textrm{GRH}_q$. 
More precisely, 
 suppose $\textrm{GRH}_q$ is true and let $x > \max\{x_q, e^4\}$. Then, 
  \begin{equation}\label{logpbound}
\log{f(x;q,1)} \leq \frac{\mathcal{F}_q - 1.2\mathcal{R}_{q,1}  + p_q(x)}{\phi(q)\sqrt{x}\log{x}} ,
  \end{equation}
  where $p_q(x)$ is given by 
  \[p_q(x) = \frac{3\mathcal{F}_q + 1.2\mathcal{R}_{q,a}}{\log{x}} + \frac{\left(1+\frac{2}{\log{x}}\right)\mathcal{G}_q}{\sqrt{x}} + \frac{0.01\phi(q) - \mathcal{B}_{q, a} - \mathcal{M}_{q, a}}{\sqrt{x}} - \left( \frac{\mathcal{M}_{q, a}}{x}- \frac{\varphi{(q)}}{2(x-1)} \right)\cdot \sqrt{x}\log{x}.\]
Each of the constants $\mathcal{F}_q$, $\mathcal{G}_q$, $\mathcal{B}_q$ and $\mathcal{M}_q$ may be computed precisely.
Observe that $p_q(x)$ is eventually positive and decreasing toward 0 as $x$ tends to infinity, and therefore will take a maximum value $\mathcal{P}_q$ on the interval $[e^{10}, \infty)$ (note that $e^{10} \approx 22027$). It follows that 
\[\log{f(x;q,1)} \leq \frac{\mathcal{F}_q - 1.2\mathcal{R}_{q,a} + \mathcal{P}_q}{\phi(q)\sqrt{x}\log{x}} \]
for $x > \max\{x_q, e^{10}\}$.  In Table \ref{computations} for each $q \in \{3,4,5,6,7,8,9,10,12,14\}$ we verify that $\mathcal{F}_q - 1.2\mathcal{R}_{q,a}+\mathcal{P}_q$ is negative, and therefore $\log{f(x;q,1)} < 0$ for $x > \max\{x_q, e^{10}\}$ under $\textrm{GRH}_q$. 

\begin{table}[ht]
\begin{tabular}{|r|r|r|r|r|r|r|r|}
\hline
\rowcolor[HTML]{EFEFEF} 
\multicolumn{1}{|c|}{\cellcolor[HTML]{EFEFEF}$q$} & \multicolumn{1}{c|}{\cellcolor[HTML]{EFEFEF}$\mathcal{F}_q$} & \multicolumn{1}{c|}{\cellcolor[HTML]{EFEFEF}$\mathcal{G}_q$} & \multicolumn{1}{c|}{\cellcolor[HTML]{EFEFEF}$\mathcal{R}_{q,a}$} & \multicolumn{1}{c|}{\cellcolor[HTML]{EFEFEF}$\mathcal{B}_{q}$} & \multicolumn{1}{c|}{\cellcolor[HTML]{EFEFEF}$\mathcal{M}_{q}$} & \multicolumn{1}{c|}{\cellcolor[HTML]{EFEFEF}$\mathcal{P}_{q}$} & \multicolumn{1}{c|}{\cellcolor[HTML]{EFEFEF}$\mathcal{F}_{q}-1.2\mathcal{R}_{q,1}+\mathcal{P}_{q}$} \\ \hline
3                                                 & 0.1594208                                                    & 0.0986123                                                    & 2                                                                & 2.2367697                                                      & 1                                                              & 0.2668522                                                      & -1.9736270                                                                                           \\ \hline
4                                                 & 0.2017594                                                    & 0.0397208                                                    & 2                                                                & 2.2744923                                                      & 1                                                              & 0.2789234                                                      & -1.9193172                                                                                           \\ \hline
5                                                 & 0.6091908                                                    & 0.2070784                                                    & 2                                                                & 2.3067140                                                      & 2                                                              & 0.3956888                                                      & -1.3951204                                                                                           \\ \hline
6                                                 & 0.1594214                                                    & 0.1177920                                                    & 2                                                                & 1.5436226                                                      & 1                                                              & 0.2717779                                                      & -1.9688008                                                                                           \\ \hline
7                                                 & 1.4141824                                                    & 0.2972734                                                    & 2                                                                & 1.7004570                                                      & 3                                                              & 0.6354003                                                      & -0.3504173                                                                                           \\ \hline
8                                                 & 0.7532641                                                    & 0.0397208                                                    & 4                                                                & 1.6412439                                                      & 2                                                              & 0.6820415                                                      & -3.3646943                                                                                           \\ \hline
9                                                 & 1.4112121                                                    & 0.0986123                                                    & 2                                                                & 2.0466109                                                      & 3                                                              & 0.6305706                                                      & -0.3582173                                                                                           \\ \hline
10                                                & 0.6091908                                                    & 0.8113486                                                    & 2                                                                & 0.9204197                                                      & 3                                                              & 0.3357980                                                      & -1.4550112                                                                                           \\ \hline
12                                                & 0.6451560                                                    & 0.5439353                                                    & 4                                                                & 1.2309413                                                      & 3                                                              & 0.5846683                                                      & -3.5701757                                                                                           \\ \hline
14                                                & 1.4141824                                                    & 0.9935082                                                    & 2                                                                & -0.3789848                                                     & 5                                                              & 0.6550409                                                      & -0.3307767                                                                                           \\ \hline
\end{tabular}
\caption{Values used to verify that $\log{f(x;q,1)} < 0$ for $x > \max\{x_q, e^{10}\}$, under $\textrm{GRH}_q$.}
\label{computations}
\end{table}

In Table \ref{pntval}, we determine the size of $x_q$. First of all using \cite[Equation (A.3)]{bennett2018} we have a constant $c_1(q)<0.4$ such that 
\[\left\lvert \theta(x;q,b) - \frac{x}{\phi(q)} \right\rvert \leq c_1(q)\sqrt{x}\]
for all $b\in \mathbb{Z}_q^{\times}$ and $1 \leq x < 10^{10}$. This implies that 
\[\frac{\theta(x^\frac{1}{2};q,b)}{x^\frac{1}{2}} > \frac{0.6}{\phi(q)}\]
for $ \left(\frac{\phi(q)c_1(q)}{0.4}\right)^4 < x < 10^{20}$. For larger $x$, \cite[Theorem 1]{ramare1996} provides $c_2(q) < 0.4$ for which 
\[\left\lvert \theta(x;q,b) - \frac{x}{\phi(q)} \right\rvert \leq c_2(q) \frac{x}{\phi(q)}\]
for all $b\in\mathbb{Z}_q^{\times}$ and $x \geq 10^{10}$. Therefore, 
\[\frac{\theta(x^\frac{1}{2};q,b)}{x^\frac{1}{2}} \geq \frac{1 - c_2(q)}{\phi(q)} > \frac{0.6}{\phi(q)},\]
for all $b\in\mathbb{Z}_q^{\times}$ and  $x \geq 10^{20}$. Hence, for $x>x_q=\left(\frac{\phi(q)c_1(q)}{0.4}\right)^4$ the inequality \eqref{xq} holds. Thus,
\[\log{f(x;q,1)} < 0\]
for $x > \max\{\lfloor x_q \rfloor, e^{10}\}$, where $\lfloor x_q \rfloor$ is given in Table \ref{pntval}. 

For $1 < x \leq \max\{\lfloor x_q \rfloor, e^{10}\}$, we have verified $\log{f(x;q,1)} < 0$ by direct computation. 
\begin{table}[h]
\begin{tabular}{|r|r|r|r||r|r|r|r|}
\hline
\rowcolor[HTML]{EFEFEF} 
\multicolumn{1}{|c|}{\cellcolor[HTML]{EFEFEF}$q$} & \multicolumn{1}{c|}{\cellcolor[HTML]{EFEFEF}$c_1(q)$} & \multicolumn{1}{c|}{\cellcolor[HTML]{EFEFEF}$c_2(q)$} & \multicolumn{1}{c||}{\cellcolor[HTML]{EFEFEF}$\lfloor{x_q}\rfloor$} & \multicolumn{1}{c|}{\cellcolor[HTML]{EFEFEF}$q$} & \multicolumn{1}{c|}{\cellcolor[HTML]{EFEFEF}$c_1(q)$} & \multicolumn{1}{c|}{\cellcolor[HTML]{EFEFEF}$c_2(q)$} & \multicolumn{1}{c|}{\cellcolor[HTML]{EFEFEF}$\lfloor{x_q}\rfloor$} \\ \hline
3   & 1.798158 & 0.002238 & 6535                  & 8   & 1.817557 & 0.002811 & 109133                \\ \hline
4   & 1.780719 & 0.002238 & 6285                  & 9   & 1.108042 & 0.003228 & 76312                 \\ \hline
5   & 1.41248  & 0.002785 & 39805                 & 10  & 1.41248  & 0.002785 & 39805                 \\ \hline
6   & 1.798158 & 0.002238 & 6535                  & 12  & 1.735501 & 0.002781 & 90720                 \\ \hline
7   & 1.116838 & 0.003248 & 78764                 & 14  & 1.105822 & 0.003248 & 75702                 \\ \hline
\end{tabular}
\caption{Values used to determine $x_q$.}
\label{pntval}
\end{table}

Therefore, for the listed values of $q$, $\textrm{GRH}_q$ implies that $\log{f(x;q,1)} < 0$ for all $x > 1$ and Theorem \ref{criterion} is established.
\end{proof}

\subsection*{Acknowledgements}
The authors would like to thank Alia Hamieh, Jean-Louis Nicolas, Timothy Trudgian, and Peng-Jie Wong for correspondence and comments on this work.

\begin{rezabib}

\bib{bennett2018}{article}{
   author={Bennett, M. A.}
   author={Martin, G.}
   author={O'Bryant, K.}
   author={Rechnitzer, A.},
   title={Explicit bounds for primes in arithmetic progressions},
   journal={arXiv preprint arXiv:1802.00085},
   date={2018},
   pages={xv+119}
}

\bib{ford2014}{article}{
   author={Ford, Kevin},
   author={Luca, Florian},
   author={Moree, Pieter},
   title={Values of the Euler $\phi$-function not divisible by a given odd
   prime, and the distribution of Euler-Kronecker constants for cyclotomic
   fields},
   journal={Math. Comp.},
   volume={83},
   date={2014},
   number={287},
   pages={1447--1476},
   issn={0025-5718},
   review={\MR{3167466}},
   doi={10.1090/S0025-5718-2013-02749-4},
}

\bib{FM}{book}{
   author={Fr\"ohlich, A.},
   author={Taylor, M. J.},
   title={Algebraic number theory},
   series={Cambridge Studies in Advanced Mathematics},
   volume={27},
   publisher={Cambridge University Press, Cambridge},
   date={1993},
   pages={xiv+355},
   isbn={0-521-43834-9},
   review={\MR{1215934}},
}

\bib{hardy2008}{book}{
   author={Hardy, G. H.},
   author={Wright, E. M.},
   title={An introduction to the theory of numbers},
   edition={6},
   note={Revised by D. R. Heath-Brown and J. H. Silverman;
   With a foreword by Andrew Wiles},
   publisher={Oxford University Press, Oxford},
   date={2008},
   pages={xxii+621},
   isbn={978-0-19-921986-5},
   review={\MR{2445243}},
}

\bib{ingham1932}{book}{
   author={Ingham, A. E.},
   title={The distribution of prime numbers},
   series={Cambridge Mathematical Library},
   note={Reprint of the 1932 original;
   With a foreword by R. C. Vaughan},
   publisher={Cambridge University Press, Cambridge},
   date={1990},
   pages={xx+114},
   isbn={0-521-39789-8},
   review={\MR{1074573}},
}

\bib{languasco2007}{article}{
   author={Languasco, A.},
   author={Zaccagnini, A.},
   title={A note on Mertens' formula for arithmetic progressions},
   journal={J. Number Theory},
   volume={127},
   date={2007},
   number={1},
   pages={37--46},
   issn={0022-314X},
   review={\MR{2351662}},
   doi={10.1016/j.jnt.2006.12.015},
}
	
\bib{languasco2009}{article}{
   author={Languasco, A.},
   author={Zaccagnini, A.},
   title={On the constant in the Mertens product for arithmetic
   progressions. II. Numerical values},
   journal={Math. Comp.},
   volume={78},
   date={2009},
   number={265},
   pages={315--326},
   issn={0025-5718},
   review={\MR{2448709}},
   doi={10.1090/S0025-5718-08-02148-0},
}

\bib{languasco2010}{article}{
   author={Languasco, A.},
   author={Zaccagnini, A.},
   title={Computing the Mertens and Meissel-Mertens constants for sums over
   arithmetic progressions},
   note={With an appendix by Karl K. Norton},
   journal={Experiment. Math.},
   volume={19},
   date={2010},
   number={3},
   pages={279--284},
   issn={1058-6458},
   review={\MR{2743571}},
   doi={10.1080/10586458.2010.10390624},
}

\bib{leveque1996}{book}{
   author={LeVeque, William J.},
   title={Fundamentals of number theory},
   note={Reprint of the 1977 original},
   publisher={Dover Publications, Inc., Mineola, NY},
   date={1996},
   pages={viii+280},
   isbn={0-486-68906-9},
   review={\MR{1382656}},
}
	
@misc{lmfdb,
  shorthand    = {LMFDB},
  author       = {The {LMFDB Collaboration}},
  title        =  {The {L}-functions and Modular Forms Database},
  howpublished = {\url{http://www.lmfdb.org}},
  year         = {2017},
  note         = {[Online; accessed 20 October 2017]},
}

@book{maple,
 title = {Maple 2017.0},
 publisher = {Maplesoft, a division of Waterloo Maple Inc., Waterloo, Ontario.},
}

\bib{mertens1874}{article}{
   author={Mertens, Franz},
   title={Ein Beitrag zur analytischen Zahlentheorie},
   language={German},
   journal={J. Reine Angew. Math.},
   volume={78},
   date={1874},
   pages={46--62},
   issn={0075-4102},
   review={\MR{1579612}},
   doi={10.1515/crll.1874.78.46},
}
	
\bib{mccurley1984}{article}{
   author={McCurley, Kevin S.},
   title={Explicit estimates for the error term in the prime number theorem
   for arithmetic progressions},
   journal={Math. Comp.},
   volume={42},
   date={1984},
   number={165},
   pages={265--285},
   issn={0025-5718},
   review={\MR{726004}},
   doi={10.2307/2007579},
}

@manual{mpmath,
  key     = {mpmath},
  author  = {Fredrik Johansson and others},
  title   = {{m}pmath: a {P}ython library for arbitrary-precision floating-point arithmetic (version 0.18)},
  note    = {{\tt http://mpmath.org/}},
  month   = {December},
  year    = {2013},
}

\bib{montgomery2007}{book}{
   author={Montgomery, Hugh L.},
   author={Vaughan, Robert C.},
   title={Multiplicative number theory. I. Classical theory},
   series={Cambridge Studies in Advanced Mathematics},
   volume={97},
   publisher={Cambridge University Press, Cambridge},
   date={2007},
   pages={xviii+552},
   isbn={978-0-521-84903-6},
   isbn={0-521-84903-9},
   review={\MR{2378655}},
}

\bib{murty2001}{book}{
   author={Murty, M. Ram},
   title={Problems in analytic number theory},
   series={Graduate Texts in Mathematics},
   volume={206},
   note={Readings in Mathematics},
   publisher={Springer-Verlag, New York},
   date={2001},
   pages={xvi+452},
   isbn={0-387-95143-1},
   review={\MR{1803093}},
   doi={10.1007/978-1-4757-3441-6},
}

\bib{narkiewicz2000}{book}{
   author={Narkiewicz, W\l adys\l aw},
   title={The development of prime number theory},
   series={Springer Monographs in Mathematics},
   note={From Euclid to Hardy and Littlewood},
   publisher={Springer-Verlag, Berlin},
   date={2000},
   pages={xii+448},
   isbn={3-540-66289-8},
   review={\MR{1756780}},
   doi={10.1007/978-3-662-13157-2},
}

\bib{nicolas1983}{article}{
   author={Nicolas, Jean-Louis},
   title={Petites valeurs de la fonction d'Euler},
   language={French, with English summary},
   journal={J. Number Theory},
   volume={17},
   date={1983},
   number={3},
   pages={375--388},
   issn={0022-314X},
   review={\MR{724536}},
   doi={10.1016/0022-314X(83)90055-0},
}
	
\bib{nicolas2012}{article}{
   author={Nicolas, Jean-Louis},
   title={Small values of the Euler function and the Riemann hypothesis},
   journal={Acta Arith.},
   volume={155},
   date={2012},
   number={3},
   pages={311--321},
   issn={0065-1036},
   review={\MR{2983456}},
   doi={10.4064/aa155-3-7},
}

\bib{platt2016}{article}{
   author={Platt, David J.},
   title={Numerical computations concerning the GRH},
   journal={Math. Comp.},
   volume={85},
   date={2016},
   number={302},
   pages={3009--3027},
   issn={0025-5718},
   review={\MR{3522979}},
   doi={10.1090/mcom/3077},
}

\bib{PR}{article}{
   author={Platt, D. J.},
   author={Ramar\'e, O.},
   title={Explicit estimates: from $\Lambda(n)$ in arithmetic progressions
   to $\Lambda(n)/n$},
   journal={Exp. Math.},
   volume={26},
   date={2017},
   number={1},
   pages={77--92},
   issn={1058-6458},
   review={\MR{3599008}},
   doi={10.1080/10586458.2015.1123124},
}

\bib{ramare1996}{article}{
   author={Ramar\'e, Olivier},
   author={Rumely, Robert},
   title={Primes in arithmetic progressions},
   journal={Math. Comp.},
   volume={65},
   date={1996},
   number={213},
   pages={397--425},
   issn={0025-5718},
   review={\MR{1320898}},
   doi={10.1090/S0025-5718-96-00669-2},
}

\bib{rosser1962}{article}{
   author={Rosser, J. Barkley},
   author={Schoenfeld, Lowell},
   title={Approximate formulas for some functions of prime numbers},
   journal={Illinois J. Math.},
   volume={6},
   date={1962},
   pages={64--94},
   issn={0019-2082},
   review={\MR{0137689}},
}

\bib{sagemath}{manual}{
      author={Developers, The~Sage},
       title={{S}agemath, the {S}age {M}athematics {S}oftware {S}ystem
  ({V}ersion 8.3)},
        date={2018},
        note={{\tt http://www.sagemath.org}},
}

\bib{williams1974}{article}{
   author={Williams, Kenneth S.},
   title={Mertens' theorem for arithmetic progressions},
   journal={J. Number Theory},
   volume={6},
   date={1974},
   pages={353--359},
   issn={0022-314X},
   review={\MR{0364137}},
   doi={10.1016/0022-314X(74)90032-8},
}
	
\end{rezabib}

\end{document}